\documentclass[11pt,a4paper,reqno]{amsart}
\usepackage{latexsym}
\usepackage{bbm,amssymb,amsthm,amsmath}
\usepackage{graphicx}
\usepackage{multirow}
\usepackage{caption}
\usepackage{bm}
\usepackage{multirow}
\usepackage{mathtools}  %%  for long arrows
\usepackage{float}  
\usepackage[pagewise]{lineno}%\linenumbers  %   for line numbers

\usepackage{xcolor}
\usepackage[backref=page]{hyperref}
\definecolor{darkgreen}{rgb}{0.5,0.25,0}
\definecolor{darkblue}{rgb}{0,0,1}
\definecolor{answerblue}{rgb}{0,0,0.75}
\hypersetup{colorlinks,breaklinks,
linkcolor=darkblue,urlcolor=darkblue,
anchorcolor=darkblue,citecolor=darkblue}

\usepackage{enumitem}

\calclayout

\theoremstyle{plain}
\newtheorem{theorem}{Theorem}[section]	% remove [section] if you do not want two level numbering for theorems, lemmas,...
\newtheorem{lemma}[theorem]{Lemma}
\newtheorem{corollary}[theorem]{Corollary}
\newtheorem{proposition}[theorem]{Proposition}

\theoremstyle{definition}
\newtheorem{definition}[theorem]{Definition}
\newtheorem{example}[theorem]{Example}

\theoremstyle{remark}
\newtheorem{remark}[theorem]{Remark}

\numberwithin{equation}{section}	% remove if you do not want two level numbering for equations

\def\C{\mathbb{C}}
\def\R{\mathbb{R}}

\def\I{\mathbb{I}}
\def\J{\mathbb{J}}

\let\mib=\boldsymbol

\def\mnu{{\mib \nu}}

\def\mxi{{\mib \xi}}
\def\meta{{\mib \eta}}

\def\mp{\mathbf{p}}
\def\mq{\mathbf{q}}

\def\vf{\mathsf{f}}
\def\vu{\mathsf{u}}
\def\vv{\mathsf{v}}

\def\ve{\mathsf{e}}

\def\vw{\mathsf{w}}

\def\mA{\mathbf{A}}

\def\mB{\mathbf{B}}

\def\mT{\mathbf{T}}

\def\lD{\mathcal{D}}
\def\lH{\mathcal{H}}
\def\lV{\mathcal{V}}
\def\lW{\mathcal{W}}
\def\lX{\mathcal{X}}
\def\lY{\mathcal{Y}}

\def\lL{\mathcal{L}}

\def\lK{\mathcal{K}}
\def\lM{\mathcal{M}}

\def\phi{\mathsf{\varphi}}

\def\dom{\operatorname{dom}}

\def\ker{\operatorname{ker}}
\def\ran{\operatorname{ran}}

\def\gr{\operatorname{gr}}

\newcommand{\mul}{\operatorname{mul}}

\def\dup#1#2#3#4{{}_{#1\!}\langle\, #2 , #3 \,\rangle_{#4}} % dual product
\def\scp#1#2{\langle\, #1 \mid #2 \,\rangle}  % scalar product
\def\iscp#1#2{[\, #1 \mid #2 \,]}  % indefinite scalar product

\begin{document}
%\today % It is good to have the date for easier later reference. We remove this in the final form.

\title[{Boundary quadruples for Friedrichs systems}]{Boundary quadruples and bijective realisations of abstract Friedrichs operators}

%\author{M.~Erceg}\address{Marko Erceg,
%	Department of Mathematics, Faculty of Science, University of Zagreb, Bijeni\v{c}ka cesta 30,
%	10000 Zagreb, Croatia}\email{maerceg@math.hr}

\author{S.~K.~Soni}\address{Sandeep Kumar Soni, Institute for Applied Mathematics, 8010 Graz, Steyrergasse 30/III, Austria}\email{soni@tugraz.at}

\subjclass[2020]{Primary 47A20; Secondary 47A06, 35F45, 46C20, 47B44}

%   47A20   Dilations, extensions and compressions of linear operators
%   47A06   Linear relations
%   35F45   Boundary value problems for systems of linear first-order PDEs
%   46C20   Spaces with indefinite inner product
%           (Kre\u{\i}n spaces, Pontryagin spaces, etc.)
%   47B44   Linear accretive and dissipative operators

\keywords{abstract Friedrichs operators, boundary quadruples,
boundary triples, linear relations, bijective realisations,
$m$-accretive realisations, Kre\u{\i}n spaces,
intrinsic boundary conditions}

\begin{abstract}
The theory of boundary quadruples and boundary triples is well-studied for symmetric and skew-symmetric operators and in general for dual-pairs. This paper adapts a suitable version for abstract Friedrichs operators and addresses the following questions: which parameters yield bijective realisations, and which parameters yield $m-$accretive realisations. We study a boundary-quadruple framework in which closed realisations are parametrised by closed relations in a boundary space. This yields the intrinsic criterion
\[T_\Theta\ \rm{is} \ \rm{bijective}\iff \lK=\Theta\dotplus \Gamma(\ker T_1) \;.\]
For bounded operator parameters $\phi:\lK_1\to \lK_0$ in the boundary space, we introduce the reference operator
\[Q_0=\Gamma_1(\Gamma_0|_{\ker T_1})^{-1}\,,\]
prove that $\|Q_0\|< 1$, and obtain the exact criterion
\[T_\phi \ \rm{is} \ \rm{bijective} \iff \I_{\lK_0}-\phi Q_0\ \rm{is} \ \rm{bijective}\;.\]
Consequently, every non-expansive parameter gives a bijective realisation with signed boundary map, which is also $m-$accretive.

An existence criterion for boundary quadruples and boundary triples is established in terms of (V)-boundary conditions. The multiplicity of $M-$operators associated with a fixed (V)-boundary condition is addressed in an explicit way and a parametrisation of such operators is given. The theory is illustrated by a first-order ordinary differential operator and by the stationary diffusion equation, where $Q_0$ is identified as a Cayley transform of the Dirichlet-to-Neumann operator.
\end{abstract}

\maketitle

%\tableofcontents	%table of content only for long papers

\section{Introduction}\label{intro}
The theory of Friedrichs systems provides a unified framework for first-order formulations of elliptic, parabolic and hyperbolic problems; see the classical work of Friedrichs~\cite{KOF}, the abstract Hilbert-space formulation in~\cite{EGC,AEM-2017}, and a nice historical exposition in~\cite{MJensen}. In the abstract setting one starts with a joint pair of minimal operators $(T_0,\widetilde T_0)$ and studies closed restrictions of the maximal operators $T_1$ and $\widetilde T_1$. The choice of a realisation is encoded by a subspace of the maximal space (the graph space). Intrinsic descriptions of such subspaces through (V)-, (X)- and (M)-boundary conditions were developed in~\cite{EGC,ABmn,ABcpde}; the corresponding classical formulations go back to~\cite{KOF,FL,PS}.

A von Neumann-type classification for abstract Friedrichs operators was obtained in~\cite{ES23}. Determined by the von Neumann-type decomposition proved in~\cite{ES22},
\[
    \lW=\lW_0\dotplus\ker T_1\dotplus\ker\widetilde T_1,
\]
the bijective realisations are parametrised by bounded operators from $\ker\widetilde T_1$ to $\ker T_1$. Natural boundary traces in differential equations, however, generally provide a different setting. In such a setting boundedness of the boundary parameter alone need not imply bijectivity. This leads to the first question addressed here:

\medskip
\noindent\emph{Given a boundary quadruple, which boundary relations and bounded operator parameters yield bijective realisations?}
\medskip

A second question concerns the intrinsic boundary conditions:

\medskip
\noindent\emph{Given a \emph{(V)}-boundary condition, how can one parametrise the corresponding {(M)}-operators in the same boundary setting in a canonical way?}
\medskip

Boundary quadruples were introduced for skew-symmetric operators in~\cite{Arendt2023}, and boundary-system and boundary-triple approaches are discussed in~\cite{SSVW15,WW18,WW20,Wegner2017}. Since a pair of abstract Friedrichs operators is equivalent to the sum of a bounded strictly positive self-adjoint operator and a skew-symmetric operator~\cite{ES23}, the same boundary form controls the extension theory. We therefore work with a boundary quadruple $(\lK_1,\lK_0,\Gamma_1,\Gamma_0)$ on the maximal space $\lW$. The boundary map
\[
    \Gamma=(\Gamma_1,\Gamma_0)^\top:\lW\to\lK_1\oplus\lK_0
\]
induces a topological isomorphism from $\lW/\lW_0$ onto the Kre\u{\i}n space $\lK_1\oplus\lK_0$.

The relation-theoretic starting point is the realisation
\[
    \lV_\Theta:=\Gamma^{-1}(\Theta),
    \qquad
    T_\Theta:=T_1|_{\lV_\Theta},
\]
where $\Theta$ is a closed linear relation in $\lK_1\oplus\lK_0$. We prove the intrinsic criterion
\[
    T_\Theta\text{ is bijective}
    \quad\Longleftrightarrow\quad
    \lK_1\oplus\lK_0
    =\Theta\dotplus\Gamma(\ker T_1).
\]
This criterion is invariant under $\J$-unitary changes of boundary coordinates.

The main result is the operator form of this condition. The zero boundary parameter gives the reference realisation
\[
    T_{\rm ref}:=T_1|_{\ker\Gamma_0},
\]
which is always bijective. Consequently,
\[
    B_0:=\Gamma_0|_{\ker T_1}:\ker T_1\to\lK_0
\]
is a topological isomorphism. We introduce
\[
    Q_0:=\Gamma_1B_0^{-1}:\lK_0\to\lK_1.
\]
For a homogeneous solution $\vv\in\ker T_1$, the identity
\[
    \Gamma_1\vv=Q_0\Gamma_0\vv
\]
shows that $Q_0$ transfers one boundary component of the homogeneous equation to the other. We prove that $\|Q_0\|<1$ and that, for every $\phi\in\mathcal L(\lK_1,\lK_0)$,
\[
    T_\phi\text{ is bijective}
    \quad\Longleftrightarrow\quad
    \I_{\lK_0}-\phi Q_0\text{ is bijective}.
\]
Thus the geometric obstruction $\Gamma(\ker T_1)$ is represented by a single bounded operator. In the canonical von Neumann quadruple one has $Q_0=0$, and the criterion reduces to the fact that every bounded parameter yields a bijective realisation. In arbitrary boundary quadruples, $Q_0$ records the additional compatibility required of the parameter.

The factor $\I-\phi Q_0$ is formally analogous to the expressions $\I-BM(\lambda)$ occurring in Weyl-function approaches to boundary triples and adjoint pairs; see, for example,~\cite{b24,bl07,dm91,DHMS06}. We do not construct a spectral family in this paper.

The present results differ from the canonical von Neumann parametrisation in~\cite{ES23}: here the boundary coordinates are fixed in advance and need not be the kernel coordinates from the von Neumann decomposition. In such arbitrary coordinates, boundedness of the parameter alone does not guarantee bijectivity; the operator $Q_0$ measures exactly the additional compatibility. Likewise, the intrinsic existence and construction results for $M$-operators in \cite[Theorem~8]{ABcpde} and \cite[Theorem~5.4 and Remark~5.5]{BES25} are translated here into the prescribed boundary coordinates, where the admissible negative spaces and a canonical associated family of $M$-operators are described explicitly.

Non-expansive parameters correspond exactly to (V)-boundary conditions, while unitary parameters correspond to self-orthogonal subspaces with respect to the boundary form. We determine precisely when a pair of subspaces satisfying (V)-boundary conditions can be realised as the kernel-pair of a boundary quadruple, and when a self-orthogonal subspace satisfying (V)-boundary condition can be realised as the kernel of a boundary triple map. For a fixed (V)-boundary condition we parametrise the compatible maximal non-positive boundary spaces by non-expansive operators $\psi$ satisfying the explicit condition that $\I_{\lK_1}-\psi\phi$ be bijective, which leads to the parametrisation of the corresponding $(M)$-operators. 

The examples are chosen to test the main criterion. For a first-order ordinary differential operator, $Q_0$ is the scalar
\[
    \exp\left(-\int_a^b\beta(s)\,ds\right),
\]
and the abstract criterion becomes the exact scalar solvability condition. For the stationary diffusion equation, the natural Dirichlet and normal traces form a boundary quadruple, and $Q_0$ is identified as a Cayley transform of the Riesz-transformed Dirichlet-to-Neumann operator. The two examples also exhibit different (M)-operators associated with the same (V)-boundary condition.

The paper is organised as follows:
Section~\ref{sec:afo} collects the required background. Boundary quadruples and their coordinate transformations are developed in Section~\ref{sec:bq}. The main bijectivity criteria are proved in Section~\ref{sec:extensions}. Boundary triples are treated in Section~\ref{sec:bt}, while Section~\ref{sec:bc} relates the intrinsic (V)- and (M)-conditions to boundary quadruples. Section~\ref{sec:examples} contains the ODE and diffusion applications.

\textbf{Notations:}
Throughout the paper all spaces are complex and all operators and relations are linear. Hilbert-space inner products are linear in the first argument and anti-linear in the second. For a Hilbert space $\lX$, the inner product and norm are denoted by $\scp{\cdot}{\cdot}_{\lX}$ and $\|\cdot\|_{\lX}$. The space of bounded operators from $\lX$ to $\lY$ is $\mathcal L(\lX,\lY)$, and $\I_{\lX}$ denotes the identity on $\lX$. We use the standard notation $\dom A$, $\ran A$, $\ker A$, $A^*$ and $\overline A$. The restriction of $A$ to $\lV$ is written $A|_{\lV}$.

The symbols $+$ and $\dotplus$ denote algebraic sum and algebraically direct sum, respectively, while $\oplus$ denotes Hilbert-space direct sum. If a Hermitian form $[\,\cdot\mid\cdot\,]$ is present, $\lM^{[\perp]}$ denotes the corresponding orthogonal complement. The form on the graph space $\lW$ is in general degenerate, with radical $\lW_0$; the induced form on $\lW/\lW_0$ is non-degenerate. The boundary spaces $\lK_1\oplus\lK_0$ are equipped with the Kre\u{\i}n-space form that is positive on $\lK_1$ and negative on $\lK_0$.

For a Hilbert space $\lX$, $\lX'$ denotes its anti-dual, and $\dup{\lX'}{\cdot}{\cdot}{\lX}$ denotes the duality pairing. An operator $C\in\mathcal L(\lX,\lY)$ is called non-expansive if $\|C\|\leq1$. A linear map between Hilbert spaces is called a \emph{unitary isomorphism} if it is a surjective linear isometry, and a \emph{topological isomorphism} if it is bounded and bijective with bounded inverse. These qualifiers are used throughout whenever the distinction is relevant.

\section{Abstract Friedrichs operators and the boundary map}\label{sec:afo}
Abstract Friedrichs operators were introduce in \cite{EGC}, which is the generalisation of the classical Friedrichs systems which was introduced by K.~O.~Friedrichs in \cite{KOF} as symmetric positive systems of first order PDEs. For the operator-theoretic description  of the abstract Friedrichs operators, see \cite{AEM-2017}. Here we recall the definition from \cite{AEM-2017}.
\begin{definition}\label{def:abstractFO}
A (densely defined) linear operator $T$ on a complex Hilbert space $\lH$
(a scalar product is denoted by
$\scp\cdot\cdot$, which we take to be anti-linear in the second entry)
is called an \emph{abstract Friedrichs operator} if there exists a
(densely defined) linear operator $\widetilde{T}_0$ on $\lH$ with the following properties:
\begin{itemize}
 \item[(T1)] $T$ and $\widetilde{T}$ have a common domain $\lD$, i.e.~$\dom T=\dom\widetilde{T}=\lD$,
 which is dense in $\lH$, satisfying
 \[
 \scp{T\phi}\psi \;=\; \scp\phi{\widetilde T\psi} \;, \qquad \phi,\psi\in\mathcal{D} \,;
 \]
 \item[(T2)] there is a constant $\lambda>0$ for which
 \[
 \|(T+\widetilde{T})\phi\| \;\leqslant\; 2\lambda\|\phi\| \;, \qquad \phi\in\mathcal{D} \,;
 \]
 \item[(T3)] there exists a constant $\mu>0$ such that
 \[
 \scp{(T+\widetilde{T})\phi}\phi \;\geqslant\; 2\mu \|\phi\|^2 \;, \qquad \phi\in\mathcal{D} \,.
 \]
\end{itemize}
The pair $(T,\widetilde{T})$ is referred to as a \emph{joint pair of abstract Friedrichs operators}
(the definition is indeed symmetric in $T$ and $\widetilde{T}$).
\end{definition}

Before moving to the main topic of the paper, 
let us briefly recall the essential properties 
of (joint pairs of) abstract Friedrichs operators, 
which we summarise in the form of a theorem.
At the same time, we introduce the notation that is used throughout the paper.
The presentation consists of two steps: first we deal with the consequences 
of conditions (T1)--(T2), and then we highlight the additional structure implied by condition (T3). 
A similar approach can be found in \cite[Theorem 2.2]{BES25}.

\begin{theorem}\label{thm:abstractFO-prop}
Let $(T_0,\widetilde{T}_0)$ be a pair of linear operators on $\lH$ satisfy {\rm (T1)} and {\rm (T2)}. Then the following holds.
\begin{enumerate}
\item[\rm{(i)}] $T_0\subseteq \widetilde{T}_0^*=:T_1$ and $\widetilde{T}_0\subseteq T_0^*=:\widetilde{T}_1$, where 
$\widetilde{T}_0^*$ and $T_0^*$ are adjoints of $\widetilde{T}_0$ and $T_0$, respectively.

\item[\rm{(ii)}] The pair of closures $(\overline{T}_0,\overline{\widetilde{T}}_0)$ satisfies {\rm (T1)--(T2)} with the same constant $\lambda$.

\item[\rm{(iii)}] $\dom \overline{T}_0=\dom\overline{\widetilde{T}}_0=:\lW_0$ and $\dom T_1=\dom\widetilde{T}_1=:\lW$.

\item[\rm{(iv)}] The graph norms $\|\cdot\|_{T_1}:=\|\cdot\|+\|T_1\cdot\|$ and $\|\cdot\|_{\widetilde T_1}
	:=\|\cdot\|+\|\widetilde T_1\cdot\|$ are equivalent, $(\lW,\|\,\cdot\,\|_{T_1})$ is a Hilbert space (the \emph{graph space})
	and $\lW_0$ is a closed subspace in it containing $\lD$.
	
\item[\rm{(v)}] The linear operator $\overline{T_0+\widetilde{T}_0}$ is everywhere defined, bounded and self-adjoint on 
$\lH$ that coincides on $\lW$ with $T_1+\widetilde{T}_1$.

\item[\rm{(vi)}] The expression 
\begin{equation}\label{eq:D1}
\dup{\lW'}{Du}{v}{\lW} \;:=\; \scp{T_1u}{v} 
- \scp{u}{\widetilde{T}_1v} \;,
\quad u,v\in\lW \,, 
\end{equation}
defines a bounded linear operator $D\in \mathcal{L}(\lW;\lW')$ that is called the \emph{boundary operator}, as $\ker D = \lW_0$.
The \emph{boundary form}
\begin{equation}\label{eq:D}
\iscp uv :=\dup{\lW'}{Du}{v}{\lW}  \;,
\quad u,v\in\lW \,, 
\end{equation}
defines an indefinite inner product on $\lW$ (cf.~\cite{Bo}) and we have $\lW^{[\perp]}=\lW_0$ and $\lW_0^{[\perp]}=\lW$, where
the $\iscp\cdot\cdot$-orthogonal complement of a set $X\subseteq \lW$
is defined by
\begin{equation*} %\label{eq:orth-compl-D}
X^{[\perp]} := \bigl\{u\in \lW : (\forall v\in X) \quad 
\iscp uv = 0\bigr\}
\end{equation*}
and it is closed in $\lW$. Moreover, $X^{[\perp][\perp]}=X$ if and only if 
$X$ is closed in $\lW$ and $\lW_0\subseteq X$. 

For future reference, let us define
\begin{equation}\label{eq:lWposneg}
\begin{aligned}
	\lW^+ &:= \bigl\{u\in\lW : \iscp{u}{u}\geq 0\bigr\} \\
	\lW^- &:= \bigl\{u\in\lW : \iscp{u}{u}\leq 0\bigr\} \,.
\end{aligned}
\end{equation}
Note that $X\subseteq X^{[\perp]}$ implies $X\subseteq \lW^+\cap\lW^-$.
\end{enumerate}

Assume, in addition, {\rm (T3)}, i.e.~$(T_0,\widetilde{T}_0)$ is 
a joint pair of abstract Friedrichs operators. Then
\begin{itemize}
\item[\rm{(vii)}] $(\overline{T}_0,\overline{\widetilde{T}}_0)$ satisfies {\rm (T3)} with the same constant $\mu$.
\item[\rm{(viii)}] A lower bound for $\overline{T_0+\widetilde{T}_0}$ is $2\mu>0$.
\item[\rm{(ix)}] We have
\begin{equation}\label{eq:von-decomposition}
\lW \;=\; \lW_0 \dotplus \ker T_1 \dotplus \ker\widetilde T_1 \;,
\end{equation}
where the sums are direct, $\lW_0 \dotplus \ker T_1 \subseteq \lW^-$, $\lW_0  \dotplus \ker\widetilde T_1  \subseteq \lW^+$ and all spaces on the right-hand side 
are pairwise $\iscp{\cdot}{\cdot}$-orthogonal. 
Moreover, the linear projections
\begin{equation}\label{eq:projections}
p_\mathrm{k} : \lW \to \ker T_1 \quad \hbox{and} \quad
p_\mathrm{\tilde k}:\lW\to \ker \widetilde{T}_1
\end{equation}
are continuous as maps $(\lW,\|\cdot\|_{T_1})\to (\lH,\|\cdot\|)$, i.e.~$p_\mathrm{k}, p_\mathrm{\tilde k}\in\lL(\lW,\lH)$. With respect to the projections \eqref{eq:projections}, the boundary map can be characterised as
\begin{equation}\label{eq:char-D}
    \iscp uv \;=\; \iscp{p_\mathrm{\tilde k}u}{p_\mathrm{\tilde k}v} + \iscp{p_\mathrm{k} u}{p_\mathrm{k} v}\;,
\quad u,v\in\lW \,.
\end{equation}
Moreover, $\ker p_\mathrm{\tilde k} + \ker p_\mathrm{ k}=\lW$ and $\ker p_\mathrm{\tilde k} \cap \ker p_\mathrm{ k}=\lW_0$.
\item[\rm{(x)}] Let $\lV$ be a subspace of the graph space $\lW$ such that
$\lW_0\subseteq\lV\subseteq \lW^+$ (see \eqref{eq:lWposneg}). Then 
$$
(\forall u\in \lV) \qquad \|T_1u\|\geq \mu\|u\| \,.
$$ 
In particular, $\overline{\ran (T_1|_\lV)}=\ran \overline{T_1|_\lV}$.

Analogously, if $\widetilde{\lV}$ is a subspace of $\lW$ such that 
$\lW_0\subseteq\widetilde{\lV}\subseteq\lW^-$, then 
$\|\widetilde T_1 v\|\geq \mu\|v\|$, $v\in\widetilde{\lV}$, 
and $\overline{\ran (\widetilde{T}_1|_{\widetilde{\lV}})}=
\ran \overline{\widetilde{T}_1|_{\widetilde \lV}}$.
\item[\rm{(xi)}] Let $\lV\subseteq\lW$ be a closed subspace (in $\lW$) containing $\lW_0$. 
Then, for a subspace $\widetilde{\lV}$ of $\lW$, the operators $T_1|_\lV$ and $\widetilde{T}_1|_{\widetilde{\lV}}$ are 
mutually adjoint, i.e.~$(T_1|_\lV)^*=\widetilde{T}_1|_{\widetilde{\lV}}$
and $(\widetilde{T}_1|_{\widetilde{\lV}})^*=T_1|_\lV$, if and only if
$\widetilde{\lV}=\lV^{[\perp]}$.
\item[\rm{(xii)}]  Let $\lV\subseteq\lW$ be a closed subspace containing $\lW_0$
such that $\lV\subseteq \lW^+$ and $\lV^{[\perp]}\subseteq\lW^-$.
Then $T_1|_\lV:\lV\to\lH$ and $\widetilde{T}_1|_{{\lV}^{[\perp]}}:\lV^{[\perp]}\to\lH$ are bijective,
i.e.~topological isomorphisms when we equip their domains with the 
graph topology, and for every $u\in\lV$ the following estimate holds:
\begin{equation}\label{eq:apriori}
	\|u\|_{T_1} \leq \Bigl(1+\frac{1}{\mu}\Bigr) \|T_1 u\| \,.
\end{equation}
The same estimate holds for $\widetilde{T}_1$ and ${\lV}^{[\perp]}$ replacing $T_1$ and $\lV$, respectively.

These bijective realisations of $T_0$ and $\widetilde{T}_0$ we call 
\emph{bijective realisations with signed boundary map}. 
\item[\rm{(xiii)}] Let $\lV\subseteq\lW$ be a closed subspace containing $\lW_0$.
Then $T_1|_\lV:\lV\to\lH$ is bijective if and only if $\lV\dotplus\ker T_1 = \lW$.
\end{itemize}
\end{theorem}
The statements i)--iv), vii) and viii) follow easily from the corresponding
assumptions (cf.~\cite{AEM-2017, EGC}).
The claims v), x) and xii) are already 
argued in the first paper on abstract Friedrichs operators \cite{EGC} for real vector spaces
(see sections 2 and 3 there), while in \cite{ABCE} 
the arguments are repeated in the complex setting.
The same applies for vi) with a remark that for a further structure 
of indefinite inner product space $(\lW, \iscp{\cdot}{\cdot})$ we refer to 
\cite{ABcpde}. 
The decomposition given in ix) is derived in \cite[Theorem 3.1]{ES22},
while for additional claims on projectors we refer to the proof of Lemma 3.5 in the 
aforementioned reference. In the same reference one can find the proof 
of part xiii) (Lemma 3.10 there).
Finally, a characterisation of mutual-adjointness, xi), is obtained 
in \cite[Theorem 9]{AEM-2017}.

\subsection{Boundary conditions for Friedrichs operators}
We recall the definitions and connections among different boundary conditions for abstract Friedrichs operators, which are mainly followed from the references \cite{ABcpde, EGC}, while a brief overview can also be found in \cite[Chapter 2.5]{Soni2024}. We assume that $(T_0,\widetilde T_0)$ is a joint pair of abstract Friedrichs operators on a Hilbert space $\lH$.

\begin{definition}[(V)-boundary conditions]\label{dfn:V1-V2}
	 A subspace $\lV$ of the graph space $\lW$ is said to satisfy (V)-\emph{boundary conditions} (jointly with $\widetilde \lV := \lV^{[\perp]}$) if the following conditions are satisfied:
	\begin{itemize}
		\item[(V1)] The boundary form has opposite signs on these spaces. More precisely, $\lV\subseteq \lW^+$ and $\widetilde\lV\subseteq \lW^-$, i.e.
		\begin{align*}
		& (\forall u\in \lV)\qquad \iscp{u}{u}\;\geq\; 0\,,\\
		& (\forall v\in \widetilde \lV) \qquad \iscp{v}{v}\;\leq\; 0\;.
		\end{align*}
		
		\item[(V2)] The subspaces $\lV, \widetilde \lV$ are mutually $\iscp{\cdot}{\cdot}$-orthogonal, i.e.
		\begin{align*}
		\lV \;=\; \widetilde \lV^{[\perp]} \quad \mathrm{and} \quad \widetilde \lV \;=\; \lV^{[\perp]}\;.
		\end{align*}
	\end{itemize}
\end{definition}
Note that the second equality in (V2) is just the definition of $\widetilde{\lV}$, and it is repeated here to emphasise symmetry in conditions for both subspaces.
\begin{remark}\label{rem:Vcond}
	Let us note that any subspace $\lV$ of $\lW$ that satisfies assumption {\rm (V)} (in pair with $\widetilde \lV := \lV^{[\perp]}$) is, by Theorem \ref{thm:abstractFO-prop}.vi),  closed and contains $\lW_0$, and thus satisfies assumptions from part xii) of the same theorem. Since the converse is also true (see again Theorem \ref{thm:abstractFO-prop}.vi)) it follows that (V)-boundary conditions correspond to realisations with signed boundary map.
	If this is the case, note also that by part xiii) of the aforementioned theorem we have $\lV\dotplus\ker T_1 = \lW$. 
    However, not all bijections described in that part are with signed boundary map %(see Example \ref{ex:ex-semigroup} below). 
\end{remark}

\begin{definition}[(M)-boundary conditions]\label{dfn:M-boundary} Let $D$ be the boundary operator. An operator $M\in \mathcal{L}(\lW;\lW')$ is said to satisfy (M)-boundary conditions if:
	\begin{itemize}
		\item[(M1)] $M$ is non-negative: 
		\begin{align*}
		(\forall u\in \lW)\qquad \Re \dup{\lW'}{Mu}{u}{\lW}\geq 0\,,
		\end{align*}
		where $\Re z$ stands for the real part of a complex number $z$;
		\item[(M2)] the graph space can be decomposed as
		\begin{align*}
		\lW=\ker(D-M)+\ker(D+M)\;.
		\end{align*}
	\end{itemize}

\end{definition}

The following result \cite[Lemma 4.1]{EGC} justifies the usage of the notion boundary operator for $M$, as well.

\begin{lemma}\label{EGC:lem4.1} Let $M\in \mathcal{L}(\lW;\lW')$ satisfy {\rm(M)}-boundary conditions. Then,
	\begin{align*}
	\ker D\;=\;\ker M\;=\;\ker M^*\ \quad \mathrm{and}\quad \ran D\;=\;\ran M\;=\;\ran M^*\;. 
	\end{align*}
\end{lemma}

The topic of equivalence between (V) and (M) boundary conditions appeared to be challenging. Eventually the following theorem was proven.

\begin{theorem}\label{lem:m-implies-v}
Let $\lW$ be the graph-space and $D$ the boundary operator.
\begin{itemize}
	\item[\rm{(a)}] 	If $M\in \mathcal{L}(\lW;\lW')$ is an operator satisfying (M)-boundary conditions, then the subspace  $\lV:=\ker(D-M)$ satisfies \emph{(V)}-boundary conditions, and $\lV^{[\perp]}=\ker(D+M^*)$.
	\item[\rm{(b)}] If $\lV$ satisfies {\rm{(V)}}-boundary conditions, then there exists an operator $M\in \mathcal{L}(\lW;\lW')$ satisfying (M)-boundary conditions such that $\lV:=\ker(D-M)$ and $\lV^{[\perp]}=\ker(D+M^*)$.
\end{itemize}
\end{theorem}

The (a) part of the above theorem was proved in \cite[Theorem 4.2]{EGC}. The converse appeared to be more challenging. In some cases, this question boils down to closedness of the subspace $\lV+\widetilde \lV$ in the graph space $\lW$ (see e.g.~\cite[Section 4]{EGC}). In \cite[Corollary 3]{ABcpde}, this problem has been addressed in full generality. We do not pursue this question further here.

\subsection{The von Neumann extension theory} 

Recently, in \cite{ES23}, the authors presented a classification theory in the spirit of the von Neumann approach, which is well-known theory for symmetric as well as skew-symmetric operators. Here we briefly recall the results of the paper \cite{ES23} (see also \cite[Chapter 3]{Soni2024}), in the context of the requirement of this manuscript.

\begin{theorem}\label{thm:von-classif}
Let $(T_0,\widetilde{T}_0)$ be a joint pair of abstract Friedrichs operators 
on $\lH$ and let $T$ be a closed realisation of $T_0$, i.e.~$T_0\subseteq T\subseteq T_1$. For a mapping $U:(\ker\widetilde{T}_1,\iscp{\cdot}{\cdot}) \to (\ker T_1,-\iscp{\cdot}{\cdot})$ we define $\lV_U:=\bigl\{u_0+U\tilde\nu + \tilde\nu : u_0\in\lW_0, \, \tilde\nu\in\ker \widetilde{T}_1\bigr\}$.
\begin{itemize}
\item[\rm{(i)}] $T$ is bijective if and only if there exists a bounded linear operator
$U:\ker\widetilde{T}_1\to\ker T_1$ such that $\dom T=\lV_U$.

\item[\rm{(ii)}]  $T$ is a bijective realisation with signed boundary map 
if and only if there exists a linear operator $U:(\ker\widetilde{T}_1,\iscp{\cdot}{\cdot}) \to (\ker T_1,-\iscp{\cdot}{\cdot})$ such that $\|U\|\leq 1$ and $\dom T=\lV_U$.

\item[\rm{(iii)}] $\dom T= \dom T^*$ if and only if there exists a unitary isomorphism
$U:(\ker\widetilde{T}_1,\iscp{\cdot}{\cdot}) \to (\ker T_1,-\iscp{\cdot}{\cdot})$ such that $\dom T=\lV_U$.

\item[\rm{(iv)}] The mapping $U\mapsto T_1|_{\lV_U}$, is a one-to-one correspondence between the classifying operators $U$ and the realisations $T$, i.e.~$\dom T$, in each of the above cases.
\end{itemize}
\end{theorem}

\section{Boundary quadruples}\label{sec:bq}
We use the boundary-quadruple formalism introduced for skew-symmetric operators in \cite{Arendt2023}; see also \cite[Section~3]{ES23} for the relation between skew-symmetric operators and abstract Friedrichs pairs. The standard relation-theoretic ingredients are recalled from \cite{bhs24} in the notation needed here. Our emphasis is different from the dissipative-extension problem considered in \cite{Arendt2023}: we seek exact criteria for bijectivity in the given boundary quadruples. For completeness, we include the elementary arguments that are used later.

\begin{definition}\label{def:BQ}
    Let $(T_0, \widetilde{T}_0)$ be a joint pair of closed abstract Friedrichs operators on a Hilbert space $\mathcal{H}$. Suppose that there exist a pair of Hilbert spaces $(\lK_1,\lK_0)$ and a pair of boundary maps $(\Gamma_1, \Gamma_0)$ satisfying the following two conditions:
    \begin{itemize}
        \item[(BQ1)] The boundary form \eqref{eq:D} can be represented by
        \begin{equation}\label{eq:BQ}
            (\forall \vu, \vv\in \lW)\qquad \iscp{\vu}{\vv} = \scp{\Gamma_1\vu}{\Gamma_1\vv}_{\lK_1} - \scp{\Gamma_0\vu }{\Gamma_0\vv}_{\lK_0}\;.
        \end{equation}

        \item[(BQ2)] The operator $\Gamma := (\Gamma_1, \Gamma_0)^\top:\lW\to \lK_1\oplus \lK_0$ is surjective.
    \end{itemize}
Then $(\lK_1, \lK_0, \Gamma_1, \Gamma_0)$ is called a \emph{boundary quadruple} for $(T_1, \widetilde{T}_1)$. 
\end{definition}

The maps in a boundary quadruple are bounded in the graph norm, and their common kernel is the minimal space.

\begin{lemma}\label{lem:cont}
    Let $(T,\widetilde{T})$ be a joint pair of abstract Friedrichs operators in $\lH$, and $(\lK_1, \lK_0, \Gamma_1, \Gamma_0)$ be a boundary quadruple for $(T_1, \widetilde{T}_1)$. Then the following statements hold:
    \begin{itemize}
        \item[(i)] $\Gamma_0:\lW\to \lK_0$ and $\Gamma_1: \lW\to \lK_1$ are bounded with respect to the graph norm.
        \item[(ii)] $\ker\Gamma =\ker\Gamma_0\cap \ker\Gamma_1 = \lW_0$.
    \end{itemize}
\end{lemma}
\begin{proof}
  \begin{itemize}
      \item[(i)] By the closed graph theorem \cite[Theorem 2.9]{Brezis} it suffices to show that $\Gamma:\lW\to \lK_1\oplus \lK_0$ is closed. Let $\vu_n\to \vu$ be a convergent sequence in $\lW$ and $\Gamma\vu_n\to (\xi_1,\xi_0)$. Since, the boundary map $D:\lW\to\lW'$ is a bounded linear operator (see Theorem \ref{thm:abstractFO-prop}(vi)), for any $\vv\in \lW$
      \begin{align*}
          \scp{\xi_1}{\Gamma_1\vv}_{\lK_1}-\scp{\xi_0}{\Gamma_0\vv}_{\lK_0}&= \lim_{n\to \infty}(\scp{\Gamma_1\vu_n}{\Gamma_1\vv}_{\lK_1}-\scp{\Gamma_0\vu_n}{\Gamma_0\vv}_{\lK_0})\\
          &= \lim_{n\to\infty}\iscp{\vu_n}{\vv}\\
          &=\iscp{\vu}{\vv}\\
          &= (\scp{\Gamma_1\vu}{\Gamma_1\vv}_{\lK_1}-\scp{\Gamma_0\vu}{\Gamma_0\vv}_{\lK_0})\,,
      \end{align*}
      which leads to 
      \begin{align*}
          (\scp{\xi_1-\Gamma_1\vu}{\Gamma_1\vv}_{\lK_1}-\scp{\xi_0-\Gamma_0\vu}{\Gamma_0\vv}_{\lK_0})=0\;.
      \end{align*}
      Since $\Gamma$ is surjective and  $\vv$ is arbitrary, we can choose $\vv$ such that $\Gamma_1\vv= \xi_1-\Gamma_1\vu$ and $\Gamma_0\vv = \Gamma_0\vu -\xi_0$, which leads to
      \begin{align*}
          \|\xi_1-\Gamma_1\vu\|_{\lK_1}^2+\|\Gamma_0\vu-\xi_0\|_{\lK_0}^2=0\;.
      \end{align*}
      which implies that $\Gamma_1\vu=\xi_1$ and $\Gamma_0\vu=\xi_0$. Hence, $\Gamma\vu = (\xi_1,\xi_0)^\top$, which completes the proof.

      \item[(ii)] From Theorem \ref{thm:abstractFO-prop}(vi), we have $\ker D = \lW_0$. If $\vu\in \lW_0$, then 
      \begin{align}\label{eq:sum-d}
           0=\dup{\lW'}{D\vu}{\vv}{\lW}= \scp{\Gamma_1\vu}{\Gamma_1\vv}_{\lK_1}-\scp{\Gamma_0\vu}{\Gamma_0\vv}_{\lK_0}\,,
      \end{align}
      for all $\vv\in \lW$. Since $\Gamma$ is surjective, we can choose $\vv$ such that 
      \begin{align*}
          \Gamma_1\vv=\Gamma_1\vu\quad {\rm{and}}\quad \Gamma_0\vv=-\Gamma_0\vu\;.
      \end{align*}
      Thus, \eqref{eq:sum-d} becomes
      \begin{align*}
          \|\Gamma_1\vu\|_{\lK_1}^2+\|\Gamma_0\vu\|_{\lK_0}^2=0\,,
      \end{align*}
      which implies that $\Gamma_0\vu=\Gamma_1\vu=0$. Hence, $\lW_0\subseteq \ker\Gamma_0\cap\ker\Gamma_1$.
      
      To prove the opposite inclusion, let $\vu\in \ker\Gamma_0\cap\ker\Gamma_1$, then for any $\vv\in \lW$
      \begin{align*}
         \dup{\lW'}{D\vu}{\vv}{\lW}= \scp{\Gamma_1\vu}{\Gamma_1\vv}_{\lK_1}-\scp{\Gamma_0\vu}{\Gamma_0\vv}_{\lK_0}=0\,,
      \end{align*}
      implying that $\vu\in \ker D = \lW_0$. Hence, $\ker\Gamma_0\cap\ker\Gamma_1\subseteq\lW_0$.
  \end{itemize}
\end{proof}

We immediately notice the following result:
\begin{corollary}\label{cor:bqv}
The pair $(\ker\Gamma_0,\ker\Gamma_1)$ satisfies the \emph{(V)}-boundary conditions. Consequently,
\begin{equation}\label{eq:reference-realisations}
    T_{\rm ref}:=T_1|_{\ker\Gamma_0}
    \qquad\text{and}\qquad
    \widetilde T_{\rm ref}:=\widetilde T_1|_{\ker\Gamma_1}
\end{equation}
are bijective realisations.
\end{corollary}

\begin{proof}
It follows immediately from \eqref{eq:BQ} that $\ker\Gamma_0$ is non-negative and $\ker\Gamma_1$ is non-positive, which proves the {\rm(V1)}-condition.
Let $\vv\in(\ker\Gamma_0)^{[\perp]}$. For every $\vu\in \ker \Gamma_0$
\begin{equation}\label{eq:bqv}
    0=\iscp{\vu}{\vv}=\scp{\Gamma_1\vu}{\Gamma_1\vv}_{\lK_1}.
\end{equation}
Surjectivity of $\Gamma$ yields $\Gamma_1(\ker\Gamma_0)=\lK_1$; indeed, for any $\xi_1\in \lK_1$, there exists $\vw\in \lW$ such that $\Gamma_1\vw=\xi_1$ and $\Gamma_0\vw=0$, proving that $\vw\in \ker\Gamma_0$ and $\xi_1\in \Gamma_1(\ker\Gamma_0)$ i.e.,$\lK_1\subseteq \Gamma_1(\ker\Gamma_0)$. The opposite inclusion is trivial, which proves the claim. 
Thus, $\Gamma_1\vv=0$, implying that $\vv\in \ker\Gamma_1$. $(\ker\Gamma_0)^{[\perp]}\subseteq \ker \Gamma_1$. The opposite inclusion directly follows from \eqref{eq:BQ}. Hence, $(\ker\Gamma_0)^{[\perp]}= \ker \Gamma_1$. The proof of $(\ker\Gamma_1)^{[\perp]}= \ker \Gamma_0$ is completely analogous. Which proves the (V2)-condition

Hence, the pair satisfies the \rm{(V)}-boundary conditions, and the bijectivity follows from Theorem~\ref{thm:abstractFO-prop}(xii).
\end{proof}

\subsection{Existence of boundary quadruples}
\begin{theorem}\label{thm:exist-bq}
    Let $(T, \widetilde{T})$ be a joint pair of abstract Friedrichs operators on $\lH$. Then there exists a boundary quadruple $(\lK_1, \lK_0, \Gamma_1, \Gamma_0)$ for $(T_1, \widetilde{T}_1)$.
\end{theorem}
\begin{proof}
    Set
\[
\lK_1=(\ker\widetilde T_1,\iscp{\cdot}{\cdot}),
\qquad
\lK_0=(\ker T_1,-\iscp{\cdot}{\cdot}),
\qquad
\Gamma_1=p_{\mathrm{\tilde k}},
\quad
\Gamma_0=p_{\mathrm{k}}.
\]
By Theorem~\ref{thm:abstractFO-prop}(ix), the signed forms make $\lK_1$ and $\lK_0$ Hilbert spaces and the projections are bounded. Formula~\eqref{eq:char-D} gives \eqref{eq:BQ}. Finally, for arbitrary $(\xi_1,\xi_0)^\top\in\lK_1\oplus\lK_0$, the vector $\vw=\xi_1+\xi_0\in \ker\widetilde{T}_1+\ker T_1\subseteq \lW$ satisfies $\Gamma\vw=(\xi_1,\xi_0)^\top$. Hence $\Gamma$ is surjective.
\end{proof}

The realisations of $T_1$ (analogously $\widetilde T_1$) can be parametrised by the relations in $\lK_1\oplus\lK_0$. Further, the realisations of interest can be parametrised by the operators from $\lK_1$ to $\lK_0$. Let us first start with the introduction of relations in $\lK_1\oplus \lK_0$ (we refer to the monograph (\cite{bhs24} for a comprehensive discussion on symmetric relations and the relations on the boundary spaces, while the core idea can be adapted here).

\begin{definition}
A linear relation in $\lK_1 \oplus \lK_0$ is a linear subspace
\[
 \Theta\subseteq\lK_1\oplus\lK_0\,,
\]
where domain and multivalued part are defined as
\[
 \dom\Theta
 :=\{\xi_1\in\lK_1:\ {\rm{for \ some\ \xi_0\in \lK_0,\ }} (\xi_1,\xi_0)^{\top}\in\Theta\}\ {\rm{and}}\ \mul\Theta
 :=\{\xi_0\in\lK_0:\ (0,\xi_0)^{\top}\in\Theta\}\;.
\]
An operator $\phi:\dom\phi\subseteq\lK_1\to\lK_0$ is identified with its graph
\[
 \gr\phi:=\{(\xi,\phi\xi)^{\top}:\xi\in\dom\phi\}.
\]
A linear relation $\Theta $ represents the graph of a linear operator if and only if
\[\mul \Theta = \{0\}\;.\]
Throughout this paper, we shall refer to a relation as a linear relation and an operator as a linear operator.
For a boundary relation $\Theta$, define
\begin{equation}\label{eq:VTheta}
 \lV:=\Gamma^{-1}(\Theta)
 =\{\vu\in\lW:\Gamma \vu\in\Theta\},
 \qquad T_{\Theta}:=T_1|_{ \lV}.
\end{equation}
\end{definition}

\begin{proposition}\label{prop:relation-param}
There exists a one-to-one correspondence between linear relations  $\Theta\subseteq\lK_1\oplus\lK_0$, and  linear subspaces $\mathcal V\subseteq\lW$ satisfying $\lW_0\subseteq\mathcal V$, given by
\begin{equation}
    \Theta\longmapsto \lV:=\Gamma^{-1}(\Theta), \quad {\rm{or}}\quad \lV\longmapsto \Gamma(\lV):=\Theta
\end{equation}
Moreover,
\begin{equation}\label{eq:closed-relation-V}
 \Theta\text{ is closed}
 \quad\iff \quad
 \mathcal V\text{ is closed in }\lW.
\end{equation}
\end{proposition}

\begin{proof}
    Let $q:\lW\to \lW/\lW_0$ be the quotient map. Since $\Gamma$ is bounded and surjective with $\ker\Gamma=\lW_0$, it induces a bounded bijection $\widehat\Gamma:\lW/\lW_0\to \lK_1\oplus\lK_0$ defined by
    \begin{equation}
         \widehat{\Gamma}(q(\vu)):=\Gamma\vu\;.
    \end{equation}
    Due to the \emph{bounded inverse theorem} \cite[Corollary 2.7]{Brezis}, $\widehat\Gamma$ is a topological isomorphism. Consequently, $\widehat\Gamma$ is a one-to-one correspondence between the linear subspaces of $\lW/\lW_0$ and the linear relations in $\lK_1\oplus\lK_0$. Hence, $\widehat\Gamma\circ q$ provides a one-to-one correspondence between the linear subspaces of $\lW$ which contain $\lW_0$ and the linear relations in $\lK_1\oplus\lK_0$, which completes the first part.

     Moreover, a subspace $\lV\subseteq \lW$ containing $\lW_0$ is closed if and only if $q(\lV)$ is closed in $\lW/\lW_0$. Since $\widehat\Gamma(q(\lV))=\Theta$, the equivalence \eqref{eq:closed-relation-V} follows.
\end{proof}

Let us now discuss when the relation $\Theta$ becomes the graph of an operator.

\begin{proposition}[When the boundary relation is an operator graph]\label{prop:graph-characterization}
Let $\mathcal V\subseteq\lW$ contain $\lW_0$, and put $\Theta=\Gamma(\mathcal V)$. Then:
\begin{enumerate}[label=\textup{(\roman*)}]
 \item $\Theta$ is the graph of an operator from a subspace of $\lK_1$ to $\lK_0$ if and only if
 \begin{equation}\label{eq:operator-intersection}
  \mathcal V\cap\ker\Gamma_1=\lW_0.
 \end{equation}
 \item $\dom\Theta=\lK_1$ if and only if
 \begin{equation}\label{eq:full-domain-sum}
  \mathcal V+\ker\Gamma_1=\lW.
 \end{equation}
\end{enumerate}
\end{proposition}

\begin{proof} \begin{itemize}
\item[(i)] A relation is a graph of an operator exactly when its multivalued part is zero. Suppose that \eqref{eq:operator-intersection} holds and $(0,\xi_0)^{\top}\in\Theta$. Since $\Gamma$ is surjective, there exists $\vu\in\mathcal V$ with $\Gamma_1\vu=0$ and $\Gamma_0\vu=\xi_0$. That is $\vu\in\mathcal V\cap\ker\Gamma_1=\lW_0$, so $\xi_0=\Gamma_0\vu=0$. Thus, $\mul\Theta=\{0\}$.

Conversely, suppose that $\mul\Theta=\{0\}$, and let $\vu\in\mathcal V\cap\ker\Gamma_1$, then $(0,\Gamma_0\vu)^{\top}\in\Theta$, which implies that $\Gamma_0\vu=0$ and $\vu\in\ker\Gamma=\lW_0$. Therefore, $\lV\cap\ker\Gamma_1\subseteq \lW_0$. The opposite inclusion is trivial.

\item[(ii)] Assume $\dom\Theta=\lK_1$. For any $\vw\in\lW$, then there exists $\vv\in\mathcal V$ with $\Gamma_1\vv=\Gamma_1\vw$ i.e.~$\vw-\vv\in\ker\Gamma_1$. Thus, $\vw=\vv+(\vw-\vv)\in \lV+\ker\Gamma_1$. The opposite inclusion is trivial.  

Conversely, suppose \eqref{eq:full-domain-sum} holds. For any $\xi_1\in\lK_1$, surjectivity of $\Gamma$ gives $\vw\in\lW$ with $\Gamma_1\vw=\xi_1$. Write $\vw=\vu+\vv$ with $\vu\in\mathcal V$ and $\vv\in\ker\Gamma_1$. It follows that $\Gamma_1\vu=\xi_1$, hence $\xi_1\in\dom\Theta$. The opposite inclusion is trivial.
\end{itemize}
\end{proof}

Now we formulate the realisations related to the operators.

\begin{theorem}\label{thm:extensions}
    Let $(T_0,\widetilde{T}_0)$ be a joint pair of closed abstract Friedrichs operators on $\lH$, and let $(\lK_1, \lK_0, \Gamma_1, \Gamma_0)$ be a boundary quadruple. Then the following assertions hold.
    \begin{itemize}
        \item[\emph{(i)}] There exists a bijective correspondence between the set of operators $\phi:\dom\phi\subseteq \lK_1\to \lK_0$ and the set of subspaces $\lV\subseteq \lW$ with $\lV\cap \ker\Gamma_1=\lW_0$. Moreover, this correspondence is given by
        \begin{equation}
            \phi \mapsto \lV_\phi:=\Gamma^{-1}(\gr \phi)=\{\vu\in \lW: \Gamma_1\vu\in \dom \phi, \Gamma_0\vu=\phi\Gamma_1\vu\}\;,
            \qquad T_\phi:=T_1|_{\lV_\phi}.
        \end{equation}
\end{itemize}

For corresponding $\phi$ and $\lV_\phi$:
\begin{itemize}
 \item[\emph{(ii)}] $\lV_\phi$ is closed if and only if $\phi$ is closed.
        \item[\emph{(iii)}] $\dom\phi=\lK_1$ if and only if $\lV_\phi + \ker\Gamma_1 = \lW$.

        \item[\emph{(iv)}] If $\phi\in\mathcal{L}(\lK_1, \lK_0)$  if and only if $\lV_\phi + \ker\Gamma_1 = \lW$.

        \item[\emph{(v)}] If $\phi$ is densely defined, then
        \begin{equation}
            T_\phi^* = \widetilde{T}_{\phi^*}\,,
        \end{equation}
        where
        \begin{equation}
           \dom  \widetilde{T}_{\phi^*}:=\{\vv\in \lW:\Gamma_0\vv\in \dom\phi^*,\, \Gamma_1\vv=\phi^*\Gamma_0\vv\}\;.
        \end{equation}
        
    \end{itemize}
\end{theorem}

\begin{proof}
    Part (i) follows from the fact that the bijective correspondence is the same correspondence as in Proposition \ref{prop:relation-param}, which by Proposition \ref{prop:graph-characterization} is restricted to the relations with zero multivalued parts. Then, part (ii) follows from the equivalence \eqref{eq:closed-relation-V} and the fact that $\phi$ is closed if and only if $\gr\phi$ is closed. Moreover, part (iii) follows from Proposition \ref{prop:graph-characterization}(ii), and part (iv) follows from part (ii) and the closed graph theorem \cite[Theorem 2.9]{Brezis}.

    For part (v), let $\vv\in \lW$, then by the definition of adjoint realisations, it follows that $\vv\in \dom T^*_\phi$ if and only if 
    \begin{equation}\label{eq:adjoint}
        (\forall \vu \in \lV_\phi) \qquad \iscp{\vu}{\vv}=\scp{T_\phi\vu}{\vv}-\scp{\vu}{T_\phi^*\vv}=0\;.
    \end{equation}
For any $\vu \in \lV_\phi$, there exists $\xi_1\in \dom \phi$ such that $\Gamma\vu=(\xi_1,\phi\xi_1)^\top$, so \eqref{eq:adjoint} is equivalent to 
\begin{equation}
    (\forall \xi_1\in \dom \phi)\qquad \scp{\xi_1}{\Gamma_1\vv}_{\lK_1}-\scp{\phi\xi_1}{\Gamma_0\vv}_{\lK_0}=0\;.
\end{equation}
which by the density of $\phi$ becomes equivalent to $\Gamma_0\vv\in \dom \phi^*$ and $\Gamma_1\vv=\phi^*\Gamma_0\vv$, which concludes part (v).

\end{proof}

\begin{remark}\label{rem:ttilde}
     Since the theory of Friedrichs operators is symmetric for $T$ and $\widetilde{T}$, the analogous results of Theorem \ref{thm:extensions} hold for $\widetilde{T}$. More precisely,
     \begin{itemize}
         \item[(i)] The extensions of $\widetilde{T}$ can be parametrised by the operators from $\lK_0$ to $\lK_1$, i.e.~
         \begin{equation}
             \dom \widetilde{T}_{\psi}=\{\vv\in \lW:\Gamma_0\vv\in \dom \psi, \Gamma_1\vv=\psi\Gamma_0\vv\}\;.
         \end{equation}
         \item[(ii)] $\dom \widetilde{T}_\psi$ is closed if and only if $\psi$ is closed.
         \item[(iii)] If $\psi$ is closed, then $\psi$ is bounded if and only if $\dom \widetilde{T}_\psi+\ker\Gamma_0=\lW$.
         \item[(iv)] If $\psi$ is densely defined, then $\widetilde{T}_\psi^*=T_{\psi^*}$.
     \end{itemize}
     Where, in all cases $\dom \widetilde{T}_\psi\cap \ker\Gamma_0=\lW_0$.
\end{remark}

Now we shall prove that all boundary quadruples related to a pair of abstract Friedrichs operators can be related to each other via certain $\J-$unitary transforms. We refer to \cite[Section 2.5]{bhs24} for the analogous results related to the symmetric relations.

\subsection{Transformations of boundary quadruples}

Let $\lK:=\lK_1\oplus\lK_0$ and equip it with
\begin{equation}\label{eq:iscp-re}
    \iscp{\mxi}{\meta}_{\lK}
    :=\scp{\xi_1}{\eta_1}_{\lK_1}
      -\scp{\xi_0}{\eta_0}_{\lK_0},
    \qquad
    \mxi=(\xi_1,\xi_0)^\top,
    \quad
    \meta=(\eta_1,\eta_0)^\top.
\end{equation}
The corresponding fundamental symmetry is
\begin{equation}\label{eq:jk}
    \J_{\lK}
    :=
    \begin{pmatrix}
        \I_{\lK_1}&0\\
        0&-\I_{\lK_0}
    \end{pmatrix},
\end{equation}
so that $\iscp{\mxi}{\meta}_{\lK}=\scp{\J_{\lK}\mxi}{\meta}_{\lK}$. Thus $(\lK,\iscp{\cdot}{\cdot}_{\lK})$ is a Kre\u{\i}n space.

\begin{definition}\label{def:iso-uni}
Let $\lK'=\lK_1'\oplus\lK_0'$ be equipped with the analogous fundamental symmetry $\J_{\lK'}$. A bounded operator $U:\lK\to\lK'$ is called $\J$-unitary if
\begin{equation}\label{eq:unitary}
    U^*\J_{\lK'}U=\J_{\lK}
    \qquad\text{and}\qquad
    U\J_{\lK}U^*=\J_{\lK'}.
\end{equation}
Equivalently, $U$ is a bijective isometry of the two Kre\u{\i}n spaces. We use standard facts about $\J$-unitary operators from~\cite[Chapter~I]{Bo}.
\end{definition}

\begin{theorem}\label{thm:transformations}
Let $(\lK_1,\lK_0,\Gamma_1,\Gamma_0)$ be a boundary quadruple for $(T_1,\widetilde T_1)$.
\begin{enumerate}[label=\textup{(\roman*)}]
\item If $U:\lK\to\lK'$ is $\J$-unitary and
\begin{equation}\label{def:unitary}
    \begin{pmatrix}\Gamma_1'\\ \Gamma_0'\end{pmatrix}
    :=U\begin{pmatrix}\Gamma_1\\ \Gamma_0\end{pmatrix},
\end{equation}
then $(\lK_1',\lK_0',\Gamma_1',\Gamma_0')$ is a boundary quadruple.
\item If $(\lK_1',\lK_0',\Gamma_1',\Gamma_0')$ is another boundary quadruple for the same pair, then there is a unique $\J$-unitary operator $U:\lK\to\lK'$ satisfying \eqref{def:unitary}.
\end{enumerate}
\end{theorem}

\begin{proof}
For \rm(i), the $\J$-unitarity of $U$ gives
\[
    \iscp{\Gamma'\vu}{\Gamma'\vv}_{\lK'}
    =\iscp{U\Gamma\vu}{U\Gamma\vv}_{\lK'}
    =\iscp{\Gamma\vu}{\Gamma\vv}_{\lK}
    =\iscp{\vu}{\vv},
\]
and the surjectivity of $\Gamma'=U\Gamma$ follows from the surjectivity of $U$ and $\Gamma$.

For \rm(ii), define $U$ by
\[
    U(\Gamma\vu):=\Gamma'\vu,
    \qquad \vu\in\lW.
\]
This is well defined because $\ker\Gamma=\lW_0=\ker\Gamma'$. It is defined on all of $\lK$, is surjective, and satisfies
\[
    \iscp{U\mxi}{U\meta}_{\lK'}
    =\iscp{\mxi}{\meta}_{\lK},
    \qquad \mxi,\meta\in\lK.
\]
Let $\widehat\Gamma:\lW/\lW_0\to\lK$ and $\widehat\Gamma':\lW/\lW_0\to\lK'$ be the topological isomorphisms induced by the two boundary maps. Then
\[
U=\widehat\Gamma'\widehat\Gamma^{-1},
\]
so $U$ and $U^{-1}$ are bounded. Since $U$ preserves the Kre\u{\i}n-space form, it is $\J$-unitary. Uniqueness follows from the surjectivity of $\Gamma$.
\end{proof}

\subsection{Effect of coordinate changes on boundary relations}\label{sec:transformations}
We show the effect of a $\J$-unitary change of coordinates on boundary relations and their realisations.
\begin{proposition}\label{prop:coord}
Suppose $\Gamma'=U\Gamma$, let $\Theta\subseteq\lK$, and set $\Theta'=U[\Theta]$. Then
\begin{equation}\label{eq:V-invariance}
 \{\vu\in\lW:\Gamma \vu\in\Theta\}
 =\{\vu\in\lW:\Gamma'\vu\in\Theta'\}.
\end{equation}
Then:
\begin{enumerate}[label={\rm{(\roman*)}}]
    \item the corresponding realisations coincide:
    \[
        T_\Theta=T'_{\Theta'}\;;
    \]
    \item $\Theta$ is closed, non-negative, non-positive, maximal
    non-negative, maximal non-positive, or hypermaximal neutral if and
    only if $\Theta'$ has the corresponding property, and
\begin{equation}\label{eq:adjoint-relation-transform}
 U[\Theta^{[\perp]}]=(\Theta')^{[\perp]}\;;
\end{equation}
    \item
    \[
        \lK=\Theta\dotplus\Gamma(\ker T_1)
        \quad\Longleftrightarrow\quad
        \lK'=\Theta'\dotplus\Gamma'(\ker T_1)\;.
    \]
\end{enumerate}
\end{proposition}
\begin{proof}
For $\vu\in\lW$,
\[
    \Gamma'\vu\in\Theta'
    \quad\Longleftrightarrow\quad
    U\Gamma\vu\in U[\Theta]
    \quad\Longleftrightarrow\quad
    \Gamma\vu\in\Theta.
\]
This proves \eqref{eq:V-invariance} and part \rm(i). Since a $\J$-unitary operator and its inverse are bounded and preserve the indefinite inner product, they preserve closedness, signs and maximality. Moreover, for $\mxi\in\Theta^{[\perp]}$ and $\meta\in\Theta$,
\[
    \iscp{U\mxi}{U\meta}_{\lK'}
    =\iscp{\mxi}{\meta}_{\lK}=0.
\]
Applying the same argument to $U^{-1}$ yields \eqref{eq:adjoint-relation-transform}, and hence also preservation of hypermaximal neutrality. It also follows that
\[
    U[\Gamma(\ker T_1)]=\Gamma'(\ker T_1),
\]
so the bijectivity of $U$ preserves both the sum and its directness. This proves \rm(iii).
\end{proof}

\begin{lemma}\label{lem:graph-rep}
The following assertions are valid.
\begin{enumerate}[label=(\roman*)]
    \item A closed subspace $\Phi\subseteq\lK$ is maximal non-negative if and only if
    \[
        \Phi=\gr\phi
    \]
    for a unique non-expansive operator
    \[
        \phi\in\mathcal L(\lK_1,\lK_0).
    \]

    \item A closed subspace $\Psi\subseteq\lK$ is maximal non-positive if and only if
    \[
        \Psi
        =
        \left\{
            \begin{pmatrix}\psi\xi_0\\ \xi_0\end{pmatrix}
            :
            \xi_0\in\lK_0
        \right\}
    \]
    for a unique non-expansive operator
    \[
        \psi\in\mathcal L(\lK_0,\lK_1).
    \]
\end{enumerate}
\end{lemma}

\begin{proof}
We prove~(i); the proof of~(ii) is analogous. Let $\Phi$ be maximal non-negative. Since
\[
    \begin{pmatrix}0\\ \xi_0\end{pmatrix}\in\Phi
    \quad\Longrightarrow\quad
    0\leq-\|\xi_0\|_{\lK_0}^2\,,
\]
it follows that $\mul \Phi =\{0\}$, consequently $\Phi$ is the graph
for a linear operator
\[
    \phi:\dom \phi\subseteq\lK_1\to\lK_0.
\]
Non-negativity gives
\[
    \|\phi\xi_1\|_{\lK_0}
    \leq
    \|\xi_1\|_{\lK_1},
    \qquad \xi_1\in\dom \phi.
\]
The domain of $\phi$ is closed. Indeed, if $\xi_n\in\dom\phi$ and $\xi_n\to\xi$ in $\lK_1$, contractivity makes $(\phi\xi_n)$ a Cauchy sequence. If $\phi\xi_n\to\eta$, the closedness of $\Phi=\gr\phi$ gives $(\xi,\eta)^\top\in\Phi$, and hence $\xi\in\dom\phi$. If $\dom\phi\neq\lK_1$, choose
$0\neq\zeta_1\in(\dom \phi)^{\perp_{\lK_1}}$, then
\[
    \Phi+
    \operatorname{span}
    \left\{
        \begin{pmatrix}\zeta_1\\0\end{pmatrix}
    \right\}
\]
is a proper non-negative extension of $\Phi$. Indeed, for $(\xi_1,\xi_0)^\top\in\Phi$ and $c\in\C$,
\begin{align*}
    \scp{\xi_1+c\zeta_1}{\xi_1+c\zeta_1}_{\lK_1}-\scp{\xi_0}{\xi_0}_{\lK_0}=\|\xi_1\|_{\lK_1}^2+|c|^2\|\zeta_1\|_{\lK_1}^2-\|\xi_0\|_{\lK_0}^2\geq 0\,,
\end{align*}
contradicting maximality. Thus, $\dom\phi=\lK_1$ and $\phi$ is non-expansive.

Conversely, let $\Phi=\gr\phi$ for a non-expansive
$\phi:\lK_1\to\lK_0$, and let $\mathcal M$ be a non-negative subspace containing $\Phi$. If
\[
    \begin{pmatrix}\eta_1\\\eta_0\end{pmatrix}\in\mathcal M,
\]
then
\[
    \begin{pmatrix}0\\\eta_0-\phi\eta_1\end{pmatrix}
    =
    \begin{pmatrix}\eta_1\\\eta_0\end{pmatrix}
    -
    \begin{pmatrix}\eta_1\\ \phi\eta_1\end{pmatrix}
    \in\mathcal M.
\]
Non-negativity forces $\eta_0=\phi\eta_1$. Hence $\mathcal M=\Phi$, proving maximality. Uniqueness is immediate.
\end{proof}

\begin{lemma}\label{lem:complement}
Let $\Phi\subseteq\lK$ be maximal non-negative, and let
$\Psi\subseteq\lK$ be a closed non-positive subspace such that
\[
    \lK=\Phi\dotplus\Psi.
\]
Then $\Psi$ is maximal non-positive.
\end{lemma}

\begin{proof}
Set $\Phi=\gr\phi$ with $\phi\in\mathcal L(\lK_1,\lK_0)$ non-expansive, as in Lemma~\ref{lem:graph-rep}. Let $\Psi_1$ be a non-positive subspace containing $\Psi$, and take $\mxi\in\Psi_1$. Write uniquely
\[
    \mxi=\mxi_\Phi+\mxi_\Psi,
    \qquad
    \mxi_\Phi\in\Phi,
    \quad
    \mxi_\Psi\in\Psi.
\]
Since $\mxi,\mxi_\Psi\in\Psi_1$, also $\mxi_\Phi=\mxi-\mxi_\Psi\in\Psi_1$. Thus $\mxi_\Phi$ belongs to the non-negative subspace $\Phi$ and to the non-positive subspace $\Psi_1$, so
\[
    \iscp{\mxi_\Phi}{\mxi_\Phi}_{\lK}=0.
\]
A neutral vector in a semidefinite subspace is orthogonal to that subspace: indeed, for any $\meta\in \Phi$ or $\meta\in \Psi_1$, we apply the sign inequality to $\mxi_\Phi+t\meta$ for real $t$, and then to $\mxi_\Phi+it\meta$, to get $\iscp{\mxi_\Phi}{\meta}=0$. Hence $\mxi_\Phi$ is orthogonal to $\Phi$ and to $\Psi_1$, and therefore it is also orthogonal to $\lK=\Phi+\Psi\subseteq \Phi+\Psi_1$. Since the Kre\u{\i}n-space form is non-degenerate, $\mxi_\Phi=0$. Consequently $\mxi=\mxi_\Psi\in\Psi$, and $\Psi_1=\Psi$. Therefore $\Psi$ is maximal non-positive.
\end{proof}

\section{Classification of realisations}\label{sec:extensions}

First we prove the following intrinsic criterion for bijective realisations.
\begin{theorem}\label{thm:intrinsic-bijectivity}
Let $\Theta\subseteq\lK:=\lK_1\oplus\lK_0$ be a closed linear relation and let
\[
    \lV_\Theta:=\Gamma^{-1}(\Theta)\;.
\]
Then the following assertions are equivalent:
\begin{itemize}
    \item[\rm(i)] $\lW=\lV_\Theta\dotplus\ker T_1$;

    \item[\rm(ii)] 
    \refstepcounter{equation}\label{eq:intrinsic-boundary-sum}%
    $\lK=\Theta\dotplus\Gamma(\ker T_1)$.
    \hfill\textup{(\theequation)}
\end{itemize}

\end{theorem}

\begin{proof} Since $\Theta$ is a linear relation, $0\in \Theta$ and hence $\lW_0=\ker\Gamma\subseteq \lV_\Theta$. Moreover, $\Gamma(\lV_\Theta)=\Theta$ by the surjectivity of $\Gamma$.

Suppose first that (i) holds. Let $\mxi\in \lK$, the due to surjectivity of $\Gamma$ there exists $\vw\in \lW$ such that $\Gamma\vw=\mxi$. By (i), there exist $\vu\in \lV_\Theta$ and $\vv\in \ker T_1$ such that 
\begin{align*}
    \vw = \vu + \vv\,,
\end{align*}
it follows that 
\begin{align*}
    \mxi = \Gamma \vu +\Gamma \vv \in \Theta + \Gamma(\ker T_1)\;.
\end{align*}
The opposite inclusion is trivial. Now, to prove that the sum is direct, let $\mxi\in \Theta \cap  \Gamma(\ker T_1)$. Since $\Gamma(\lV_\Theta)=\Theta$, there exists $\vu\in \lV_\Theta$ such that $\Gamma\vu=\mxi$. Also, there exists $\vv\in \ker T_1$ such that $\Gamma \vv = \mxi$. Then
\begin{align*}
    \vu-\vv\in \ker\Gamma=\lW_0\subseteq \lV_\Theta\,,
\end{align*}
which implies that $\vv\in \lV_\Theta$. Thus
\begin{align*}
    \vv \in \lV_\Theta\cap \ker T_1=\{0\}\,,
\end{align*}
the equality is due to the assumption. Hence, $\mxi=\Gamma\vv=0$, proving (ii).

Conversely, assume (ii). Let $\vw\in \lW$, then there exist $\mxi \in \Theta$ and $\vv\in \ker T_1$ such that 
\begin{align*}
    \Gamma\vw = \mxi +\Gamma\vv\;.
\end{align*}
Since $\Gamma(\lV_\Theta)=\Theta$, there exists $\vu\in \lV_\Theta$ such that $\Gamma\vu=\mxi$. Then $\vw-\vu-\vv\in \ker\Gamma=\lW_0\subseteq \lV_\Theta$, and therefore
\begin{align*}
    \vw=(\vu+\vw-\vu-\vv)+\vv\in \lV_\Theta +\ker T_1\;.
\end{align*}
The opposite inclusion is trivial. It remains to prove that the sum is direct. Let $\vv\in\lV_\Theta\cap \ker T_1$, then
\begin{align*}
    \Gamma\vv\in \Theta\cap \Gamma(\ker T_1)=\{0\}\;.
\end{align*}
Thus, $\vv\in \ker\Gamma\cap\ker T_1=\lW_0\cap\ker T_1=\{0\}$, which proves (i).

\end{proof}

\begin{corollary}\label{cor:classification}
Let $\lV\subseteq\lW$ be a closed subspace containing $\lW_0$, and put $\Theta:=\Gamma(\lV)$. Then
\[
    T_1|_{\lV}\text{ is bijective}
    \quad\Longleftrightarrow\quad
    \lK=\Theta\dotplus\Gamma(\ker T_1).
\]
\end{corollary}
\begin{proof}
    The proof follows from Theorem~\ref{thm:abstractFO-prop}(xiii) and Theorem \ref{thm:intrinsic-bijectivity} 
\end{proof}

We shall use the reference operators from Corollary \ref{cor:bqv} to classify the desired realisations.

\begin{proposition}\label{prop:zero-reference-map} Set
\begin{equation}\label{eq:B0-def}
    B_0:=\Gamma_0|_{\ker T_1}:\ker T_1\to\lK_0.
\end{equation}
Then the following assertions hold.
\begin{itemize}
    \item[{\rm(i)}] $B_0:\ker T_1\to \lK_0$ is bounded and bijective;
    \item[{\rm(ii)}]  The operator
\begin{equation}\label{eq:Q0-def}
    Q_0:=\Gamma_1B_0^{-1}:\lK_0\to\lK_1
\end{equation}
is bounded and satisfies
\begin{equation}\label{eq:Q0-strict}
    \|Q_0\|<1.
\end{equation}
\end{itemize}
\end{proposition}

\begin{proof}
By Corollary~\ref{cor:bqv} and Theorem~\ref{thm:abstractFO-prop}(xiii),
\[
    \lW=\ker\Gamma_0\dotplus\ker T_1.
\]
Now, let $\vv\in \ker T_1$ such that $B_0\vv=0$, then $\Gamma_0\vv=0$; which implies that $\vv\in \ker T_1\cap \ker\Gamma_0=\{0\}$. Thus, $\vv=0$, proving that $B_0$ is injective.
To prove surjectivity, let $\xi_0\in\lK_0$. By the surjectivity of $\Gamma$, choose $\vw\in\lW$ such that
\[
    \Gamma_1\vw=0,
    \qquad
    \Gamma_0\vw=\xi_0.
\]
Write $\vw=\vu+\vv$ with $\vu\in\ker\Gamma_0$ and $\vv\in\ker T_1$. Then $B_0\vv=\xi_0$. Thus $B_0$ is bijective. It is bounded by Lemma~\ref{lem:cont}, and its inverse is bounded by the bounded inverse theorem; cf.~\cite[Corollary~2.7]{Brezis}. Since $\Gamma_1$ is also bounded (see Lemma \ref{lem:cont}(i)), it follows $Q_0$ is bounded.

It remains to prove \eqref{eq:Q0-strict}. If $\ker T_1=\{0\}$, then $\lK_0=\{0\}$ and the assertion is trivial. Otherwise, let $C_0$ denotes the norm of the bounded restriction $\Gamma_0|_{\ker T_1}$, where $\ker T_1$ is equipped with the norm of $\lH$ (in $\ker T_1$ we have $\|\cdot\|_\lW=\|\cdot\|_\lH$). For $\vv\in\ker T_1$, Theorem~\ref{thm:abstractFO-prop}(v), (viii) and \eqref{eq:BQ} imply
\begin{align*}
    \|\Gamma_0\vv\|_{\lK_0}^2-\|\Gamma_1\vv\|_{\lK_1}^2=-\iscp{\vv}{\vv}=\scp{\vv}{\widetilde T_1\vv}=\scp{\vv}{(T_1+\widetilde T_1)\vv}\geq 2\mu\|\vv\|_{\lH}^2\;.
\end{align*}
Since $\|\Gamma_0\vv\|_{\lK_0}\leq C_0\|\vv\|_{\lH}$, we obtain
\[
    \|\Gamma_1\vv\|_{\lK_1}^2
    \leq
    \left(1-\frac{2\mu}{C_0^2}\right)
    \|\Gamma_0\vv\|_{\lK_0}^2.
\]
The preceding estimate also implies $C_0^2\geq2\mu$. Set
\[
    q_0:=\left(1-\frac{2\mu}{C_0^2}\right)^{1/2}\in[0,1).
\]
Since $\vv\in\ker T_1$ is arbitrary and $B_0:\ker T_1\to\lK_0$ is bijective, applying the estimate to $\vv=B_0^{-1}\xi_0$ gives
\[
    \|Q_0\xi_0\|_{\lK_1}\leq q_0\|\xi_0\|_{\lK_0},
    \qquad \xi_0\in\lK_0,
\]
and hence \eqref{eq:Q0-strict}.
\end{proof}

We now obtain the criterion of bijectivity of the realisations in terms of the parameters related the boundary quadruple.

\begin{theorem}\label{thm:Q0-criterion}
Let $\phi\in\mathcal L(\lK_1,\lK_0)$ and put
\[
    \lV_\phi=\ker(\Gamma_0-\phi\Gamma_1),
    \qquad
    T_\phi=T_1|_{\lV_\phi}.
\]
Then the following assertions are equivalent:
\begin{itemize}
    \item[\rm(i)] $T_\phi$ is bijective;
    \item[\rm(ii)] the operator
    \begin{equation}\label{eq:Bphi-def}
        B_\phi:=\left.(\Gamma_0-\phi\Gamma_1)\right|_{\ker T_1}:
        \ker T_1\to\lK_0
    \end{equation}
    is bijective;
    \item[\rm(iii)] the operator
    \begin{equation}\label{eq:Q0-admissibility}
        \I_{\lK_0}-\phi Q_0:\lK_0\to\lK_0
    \end{equation}
    is bijective.
\end{itemize}
Moreover,
\begin{equation}\label{eq:Bphi-factorisation}
    B_\phi=(\I_{\lK_0}-\phi Q_0)B_0.
\end{equation}
\end{theorem}

\begin{proof} To prove the equivalence between (i) and (ii), by Theorem~\ref{thm:abstractFO-prop}(xiii)  it is enough to prove that
\[
    B_\phi\text{ is bijective}
    \quad\Longleftrightarrow\quad
    \lW=\lV_\phi\dotplus\ker T_1.
\]
Since $\Gamma$ is surjective, $\Gamma_0-\phi\Gamma_1:\lW\to\lK_0$ is surjective. Indeed, if $\xi_0\in \lK_0$, then there exists $\vu\in \lW$ such that $\Gamma\vu = (0, \xi_0)^\top$; which implies that $(\Gamma_0-\phi\Gamma_1)\vu=\xi_0$, proving the claim. Furthermore,
\[
    \ker B_\phi=\lV_\phi\cap\ker T_1.
\]
Thus, $B_\phi$ is injective if and only if the sum $\lV_\phi+\ker T_1$ is direct.

Suppose first that $B_\phi$ is surjective. For $\vw\in\lW$, there exists $\vv\in\ker T_1$ such that
\[
    B_\phi \vv=\Gamma_0\vw-\phi\Gamma_1\vw.
\]
Then $\vw-\vv\in\lV_\phi$, and hence
\[
    \lW=\lV_\phi+\ker T_1.
\]
Conversely, if this sum equals $\lW$, then for any $\eta_0\in\lK_0$ choose $\vw\in\lW$ with $\Gamma_0-\phi\Gamma_1\vw=\eta_0$ and write $\vw=\vu+\vv$, where $\vu\in\lV_\phi$ and $\vv\in\ker T_1$. It follows that $B_\phi \vv=\eta_0$. Therefore,
\[
    B_\phi\text{ is bijective}
    \quad\Longleftrightarrow\quad
    \lW=\lV_\phi\dotplus\ker T_1\,,
\]
which proves the equivalence of \rm(i) and \rm(ii).

For $\vv\in\ker T_1$, the definition of $Q_0$ yields
\[
    \Gamma_1\vv=Q_0\Gamma_0\vv=Q_0B_0\vv.
\]
Consequently,
\begin{align*}
    B_\phi \vv =\Gamma_0\vv-\phi\Gamma_1\vv =(\I_{\lK_0}-\phi Q_0)B_0\vv,
\end{align*}
which proves \eqref{eq:Bphi-factorisation}. Since $B_0$ is bijective, \rm(ii) and \rm(iii) are equivalent.
\end{proof}

\begin{corollary}\label{cor:neumann-Q0}
Let $\phi\in\mathcal L(\lK_1,\lK_0)$. If
\[
    \|\phi Q_0\|<1,
\]
then $T_\phi$ is bijective. In particular, every non-expansive operator $\phi:\lK_1\to\lK_0$ gives a bijective realisation.
\end{corollary}

\begin{proof}
The first assertion follows from the Neumann series for $\I_{\lK_0}-\phi Q_0:\lK_0\to \lK_0$ (see \cite[Exercise 6.14]{Brezis}). If $\phi$ is non-expansive, then Proposition~\ref{prop:zero-reference-map} gives
\[
    \|\phi Q_0\|\leq\|Q_0\|<1.
\]
The conclusion follows from Theorem~\ref{thm:Q0-criterion}.
\end{proof}

\begin{remark} If $Q_0=0$, then the boundedness alone of the parameter is enough for bijective realisation. For the canonical boundary quadruple constructed in Theorem~\ref{thm:exist-bq}, one has $Q_0=0$. Hence, Theorem~\ref{thm:Q0-criterion} recovers the von-Neumann statement in Theorem~\ref{thm:von-classif}: every bounded parameter gives a bijective realisation. For a general boundary quadruple, $Q_0$ encodes the necessary information about the homogeneous equation.
\end{remark}

We call $\phi\in \mathcal{L}(\lK_1,\lK_0)$ \emph{$Q_0-$admissible} if $\I_{\lK_0}-\phi Q_0:\lK_0\to \lK_0$ is bijective. We conclude this section with the classification of the signed and neutral boundary conditions. 

\begin{theorem}\label{thm:classification}
Let $(T_0,\widetilde T_0)$ be a joint pair of closed abstract Friedrichs operators on $\lH$, let $(\lK_1,\lK_0,\Gamma_1,\Gamma_0)$ be a boundary quadruple, and let $\phi\in\mathcal L(\lK_1,\lK_0)$. Then:
\begin{itemize}
    \item[\rm(i)] $T_\phi$ is bijective if and only if $\phi$ is $Q_0$-admissible;
    \item[\rm(ii)] the pair $(\lV_\phi,\lV_\phi^{[\perp]})$ satisfies the \emph{(V)}-boundary conditions if and only if $\phi$ is non-expansive;
    \item[\rm(iii)] $\lV_\phi=\lV_\phi^{[\perp]}$ if and only if $\phi$ is unitary.
\end{itemize}
In particular, every non-expansive, and hence every unitary, parameter gives a bijective realisation.
\end{theorem}

\begin{proof}
Part \rm(i) is Theorem~\ref{thm:Q0-criterion}. 

For $\vu\in\lV_\phi$,
\[
    \iscp{\vu}{\vu}
    =\|\Gamma_1\vu\|_{\lK_1}^2
     -\|\phi\Gamma_1\vu\|_{\lK_0}^2\geq 0\;.
\]

The complement of $\gr \phi$ in $\lK$ is
\begin{equation}\label{eq:graph-orthogonal}
    (\gr\phi)^{[\perp]}
    =\left\{(\phi^*\eta_0,\eta_0)^\top:\eta_0\in\lK_0\right\}.
\end{equation}
Indeed, $(\eta_1,\eta_0)^\top\in(\gr\phi)^{[\perp]}$ if and only if
\[
    \scp{\xi_1}{\eta_1}_{\lK_1}
    -\scp{\phi\xi_1}{\eta_0}_{\lK_0}=0
    \qquad(\xi_1\in\lK_1=\dom \phi),
\]
which is equivalent to $\eta_1=\phi^*\eta_0$. By Proposition~\ref{prop:relation-param},
\[
    \lV_\phi^{[\perp]}
    =\Gamma^{-1}((\gr\phi)^{[\perp]}).
\]
So, $\lV_\phi^{[\perp]}$ is negative if and only if $\|\phi^*\|\leq1$, which is equivalent to $\|\phi\|\leq1$. This proves \rm(ii).

Finally, \eqref{eq:graph-orthogonal} shows that
\[
    \lV_\phi=\lV_\phi^{[\perp]}
    \quad\Longleftrightarrow\quad
    \gr\phi=(\gr\phi)^{[\perp]}.
\]
If $\phi$ is unitary, the second equality is immediate. Conversely, this equality implies first that
\[
    \|\phi\xi_1\|_{\lK_0}=\|\xi_1\|_{\lK_1}
    \qquad(\xi_1\in\lK_1),
\]
and then, by using every vector $(\phi^*\eta_0,\eta_0)^\top$ in the right-hand side of \eqref{eq:graph-orthogonal}, it follows that $\phi\phi^*=\I_{\lK_0}$. Hence $\phi$ is unitary. This proves \rm(iii). The final assertion follows from Corollary~\ref{cor:neumann-Q0}.
\end{proof}

\section{Boundary triples and extension theory}\label{sec:bt}

\begin{definition}\label{def:BT}
    Let $(T_0, \widetilde{T}_0)$ be a joint pair of closed abstract Friedrichs operators on a Hilbert space $\lH$. Suppose that there exist a Hilbert space $\lX$ and a pair of boundary maps $(\gamma_1,\gamma_0)$ satisfying the following two conditions:
    \begin{itemize}
        \item[(BT1)] The boundary form \eqref{eq:D} can be represented as
        \begin{equation}\label{eq:BT}
            (\forall\,\vu,\vv\in\lW)\qquad
            \iscp{\vu}{\vv}
            =
            \scp{\gamma_1\vu}{\gamma_0\vv}_{\lX}
            +
            \scp{\gamma_0\vu}{\gamma_1\vv}_{\lX}.
        \end{equation}

        \item[(BT2)] The operator
        \[
            \gamma:=
            \begin{pmatrix}
                \gamma_1\\
                \gamma_0
            \end{pmatrix}
            :\lW\to\lX\oplus\lX
        \]
        is surjective.
    \end{itemize}
    Then $(\lX, \gamma_1, \gamma_0)$ is called a \emph{boundary triple} for $(T_1,\widetilde{T}_1)$.
\end{definition}

\begin{remark} For the usual boundary triples of symmetric operators, Green's identity contains the skew-Hermitian combination $\scp{\gamma_1\vu}{\gamma_0\vv}-\scp{\gamma_0\vu}{\gamma_1\vv}$. The plus sign in \eqref{eq:BT} reflects the skew-symmetric core of a joint pair of abstract Friedrichs operators; see \cite{ES23}. 
\end{remark}

\begin{lemma}\label{lem:bt-cont}
For every boundary triple $(\lX,\gamma_1,\gamma_0)$, the maps $\gamma_0,\gamma_1:\lW\to\lX$ are bounded in the graph norm and
\[
\ker\gamma_0\cap\ker\gamma_1=\lW_0.
\]
\end{lemma}
\begin{proof}
The invertible change of variables
\[
\Gamma_1:=\frac{\gamma_1+\gamma_0}{\sqrt2},
\qquad
\Gamma_0:=\frac{\gamma_1-\gamma_0}{\sqrt2}
\]
turns the boundary triple into the boundary quadruple
$
(\lX,\lX,\Gamma_1,\Gamma_0)$.
The assertion follows from Lemma~\ref{lem:cont}.
\end{proof}

\subsection{Existence}

The existence of a boundary triple is subject to the unitary isomorphism of the kernel spaces appearing in the von Neumann decomposition. More precisely, we have the following result.

\begin{theorem}\label{thm:exist-bt}
    Let $(T,\widetilde{T})$ be a joint pair of abstract Friedrichs operators on $\lH$ such that the Hilbert spaces
    \[
        (\ker T_1,-\iscp{\cdot}{\cdot})
        \quad\text{and}\quad
        (\ker\widetilde{T}_1,\iscp{\cdot}{\cdot})
    \]
    are unitarily isomorphic. Then there exists a boundary triple $(\lX,\gamma_1, \gamma_0)$ for $(T_1,\widetilde{T}_1)$.
\end{theorem}

\begin{remark}\label{rem:isomorhism=unitary}
    The terminology is as fixed in the notation section. The hypothesis of Theorem~\ref{thm:exist-bt} may equivalently be formulated by requiring a topological isomorphism between the two kernel Hilbert spaces, since the existence of a topological isomorphism between Hilbert spaces implies the existence of a unitary isomorphism; see \cite[Remark~3.13]{ES23}.
\end{remark}

\begin{proof}
    Let
    \[
        \theta:(\ker\widetilde{T}_1,\iscp{\cdot}{\cdot})
        \to
        (\ker T_1,-\iscp{\cdot}{\cdot})
    \]
    be a unitary isomorphism, and set
    \[
        \lX:=(\ker T_1,-\iscp{\cdot}{\cdot}).
    \]
    Define $\gamma_0,\gamma_1:\lW\to\lX$ by
    \begin{equation}\label{eq:canonical-bt}
        \gamma_0
        :=
        \frac{1}{\sqrt{2}}\bigl(\theta\circ p_{\mathrm{\tilde k}}-p_{\mathrm{k}}\bigr),
        \qquad
        \gamma_1
        :=
        \frac{1}{\sqrt{2}}\bigl(\theta\circ p_{\mathrm{\tilde k}}+p_{\mathrm{k}}\bigr).
    \end{equation}
    For $\vu,\vv\in\lW$, a direct computation gives
    \begin{align*}
        \scp{\gamma_1\vu}{\gamma_0\vv}_{\lX}
        +\scp{\gamma_0\vu}{\gamma_1\vv}_{\lX}
        =
        \scp{\theta p_{\mathrm{\tilde k}}\vu}
             {\theta p_{\mathrm{\tilde k}}\vv}_{\lX}
        -
        \scp{p_{\mathrm{k}}\vu}{p_{\mathrm{k}}\vv}_{\lX}
        =        \iscp{p_{\mathrm{\tilde k}}\vu}{p_{\mathrm{\tilde k}}\vv}
        +
        \iscp{p_{\mathrm{k}}\vu}{p_{\mathrm{k}}\vv}\;.
    \end{align*}
    Hence, from \eqref{eq:char-D}, {\rm (BT1)} holds.

    It remains to prove {\rm (BT2)}. Let $\xi_1,\xi_0\in\lX$ and set
    \[
        \tilde\eta
        :=
        \theta^{-1}\left(\frac{\xi_1+\xi_0}{\sqrt{2}}\right)\in \ker\widetilde T_1,
        \qquad
        \eta
        :=
        \frac{\xi_1-\xi_0}{\sqrt{2}}\in \ker T_1.
    \]
    By the von Neumann decomposition \eqref{eq:von-decomposition}, there exists $\vw\in \lW$ such that $\vw= \tilde\eta+\eta$, and
    therefore,
    \[
        \gamma_1\vw=\xi_1,
        \qquad
        \gamma_0\vw=\xi_0.
    \]
    Hence, $\gamma:\lW\to\lX\oplus\lX$ is surjective.
\end{proof}

\begin{remark}\label{rem:btv}
    Let $(\gamma_0,\gamma_1,\lX)$ be a boundary triple for $(T_1,\widetilde{T}_1)$. Then
    \begin{equation}\label{eq:neutral-kernels}
        (\ker\gamma_0)^{[\perp]}=\ker\gamma_0,
        \qquad
        (\ker\gamma_1)^{[\perp]}=\ker\gamma_1.
    \end{equation}
    Indeed, if $\vu,\vv\in\ker\gamma_0$, then \eqref{eq:BT} gives
    \[
        \iscp{\vu}{\vv}=0,
    \]
    and hence
    \[
        \ker\gamma_0\subseteq(\ker\gamma_0)^{[\perp]}.
    \]
    Conversely, let $\vu\in(\ker\gamma_0)^{[\perp]}$. For any $\xi\in\lX$, the surjectivity of $\gamma$ yields $\vv\in\lW$ such that
    \[
        \gamma_1\vv=\xi,
        \qquad
        \gamma_0\vv=0.
    \]
    Thus, $\vv\in\ker\gamma_0$, and therefore
    \[
        0
        =
        \iscp{\vu}{\vv}
        =
        \scp{\gamma_0\vu}{\xi}_{\lX}.
    \]
    Since $\xi\in\lX$ is arbitrary, $\gamma_0\vu=0$. This proves the first equality in \eqref{eq:neutral-kernels}. The second equality follows analogously.

    Consequently, both $\ker\gamma_0$ and $\ker\gamma_1$ are self-orthogonal subspaces of $\lW$ with respect to the indefinite inner product. In particular, $(\ker\gamma_0,\ker\gamma_0)$ and $(\ker\gamma_1,\ker\gamma_1)$ satisfy the {\rm (V)}-boundary conditions.
\end{remark}

The existence condition for boundary triples can equivalently be expressed in terms of a boundary quadruple.

\begin{corollary}\label{cor:bt}
    Let $(T_0,\widetilde{T}_0)$ be a joint pair of closed abstract Friedrichs operators on $\lH$, and let $(\lK_1,\lK_0,\Gamma_1,\Gamma_0)$ be a boundary quadruple. Then the following assertions are equivalent:
    \begin{itemize}
        \item[\emph{(i)}] The Hilbert spaces
        \[
            (\ker T_1,-\iscp{\cdot}{\cdot})
            \quad\text{and}\quad
            (\ker\widetilde{T}_1,\iscp{\cdot}{\cdot})
        \]
        are unitarily isomorphic;

        \item[\emph{(ii)}] the Hilbert spaces $\lK_1$ and $\lK_0$ are unitarily isomorphic;

        \item[\emph{(iii)}] there exists a unitary isomorphism $\theta:\lK_1\to\lK_0$ such that, with $\lX:=\lK_0$, the maps
        \begin{equation}\label{eq:cor-bt}
            \gamma_0
            :=
            \frac{1}{\sqrt{2}}\bigl(\theta\circ\Gamma_1-\Gamma_0\bigr),
            \qquad
            \gamma_1
            :=
            \frac{1}{\sqrt{2}}\bigl(\theta\circ\Gamma_1+\Gamma_0\bigr)
        \end{equation}
        form a boundary triple $(\lX, \gamma_1, \gamma_0)$ for $(T_1,\widetilde{T}_1)$.
    \end{itemize}
\end{corollary}

\begin{proof}
    The canonical boundary quadruple from Theorem~\ref{thm:exist-bq} and the given boundary quadruple $(\lK_1,\lK_0,\Gamma_1,\Gamma_0)$ are related by a $\J$-unitary operator. Hence the positive and negative indices of the corresponding Kre\u{\i}n spaces coincide; see, for example, \cite[Corollary~11.9, Chapter~I]{Bo}. Consequently, there exist unitary isomorphisms
    \[
        U_+:\lK_1\to(\ker\widetilde T_1,\iscp{\cdot}{\cdot}),
        \qquad
        U_-:\lK_0\to(\ker T_1,-\iscp{\cdot}{\cdot}).
    \]
    Therefore $\lK_1$ and $\lK_0$ are unitarily isomorphic if and only if $(\ker\widetilde T_1,\iscp{\cdot}{\cdot})$ and $(\ker T_1,-\iscp{\cdot}{\cdot})$ are unitarily isomorphic. Indeed, if $\theta:\lK_1\to\lK_0$ is a unitary isomorphism, then $U_-\theta U_+^{-1}$ is a unitary isomorphism between the kernel spaces; the converse follows in the same way.

    Assume {\rm (ii)}, and let $\theta:\lK_1\to\lK_0$ be a unitary isomorphism. Set $\lX:=\lK_0$ and define $\gamma_0,\gamma_1$ by \eqref{eq:cor-bt}. A direct computation using the fact that $\theta$ is unitary, gives \eqref{eq:BT}. Moreover,
    \[
        \begin{pmatrix}
            \gamma_1\\
            \gamma_0
        \end{pmatrix}
        =
        \frac{1}{\sqrt{2}}
        \begin{pmatrix}
            \theta & \I_{\lK_0}\\
            \theta & -\I_{\lK_0}
        \end{pmatrix}
        \begin{pmatrix}
            \Gamma_1\\
            \Gamma_0
        \end{pmatrix}.
    \]
    The operator
    \[
        \frac{1}{\sqrt{2}}
        \begin{pmatrix}
            \theta & \I_{\lK_0}\\
            \theta & -\I_{\lK_0}
        \end{pmatrix}
        :\lK_1\oplus\lK_0\to\lX\oplus\lX
    \]
    is bijective, with inverse
    \[
        \frac{1}{\sqrt{2}}
        \begin{pmatrix}
            \theta^{-1} & \theta^{-1}\\
            \I_{\lK_0} & -\I_{\lK_0}
        \end{pmatrix}.
    \]
    Since $\Gamma:\lW\to\lK_1\oplus\lK_0$ is surjective, it follows that $\gamma:\lW\to\lX\oplus\lX$ is surjective. Hence, {\rm (BT2)} holds, and {\rm (ii)} implies {\rm (iii)}. The implication {\rm (iii)}$\Rightarrow${\rm (ii)} is immediate.
\end{proof}

\begin{remark}
With $\theta$ as in Corollary~\ref{cor:bt}, a boundary-quadruple condition
\[
    \Gamma_0\vu=(\phi\theta)\Gamma_1\vu,
    \qquad \phi\in\mathcal L(\lX),
\]
is equivalent to
\[
    (\I_{\lX}-\phi)\gamma_1\vu
    -(\I_{\lX}+\phi)\gamma_0\vu=0.
\]
Thus the boundary-triple parametrisation is a Cayley-type rewriting of the boundary quadruple parametrisation.
\end{remark}
\section{An overview of different boundary conditions}\label{sec:bc}

\subsection{(V)-boundary conditions}

Theorem \ref{thm:classification}(ii) together with Theorem \ref{thm:extensions}(v) yields the following characterisation of (V)-boundary conditions.
\begin{theorem}\label{thm:v-classif}
Let $(T_0,\widetilde{T}_0)$ be a joint pair of closed abstract Friedrichs operators on $\lH$, let $(\lK_1,\lK_0,\Gamma_1,\Gamma_0)$ be a boundary quadruple, and let $\lV\subseteq\lW$ be a closed subspace containing $\lW_0$. Put $\widetilde\lV:=\lV^{[\perp]}$. Then the following assertions are equivalent:
\begin{itemize}
    \item[{\rm(i)}] The pair $(\lV,\widetilde\lV)$ satisfies the \emph{(V)}-boundary conditions;
    \item[{\rm(ii)}] there exists a unique non-expansive operator $\phi:\lK_1\to\lK_0$ such that
    \begin{equation}\label{eq:v-classif}
        \lV=\ker(\phi\Gamma_1-\Gamma_0),
        \qquad
        \widetilde\lV=\ker(\Gamma_1-\phi^*\Gamma_0).
    \end{equation}
\end{itemize}
\end{theorem}
\begin{proof}
If (i) holds, then $\Gamma(\lV)$ is a maximal non-negative subspace of $\lK$, hence by Lemma~\ref{lem:graph-rep} it is the graph of a unique non-expansive parameter $\phi:\lK_1\to\lK_0$. Therefore $\lV=\ker(\phi\Gamma_1-\Gamma_0)$. Proposition~\ref{prop:coord}, or a direct computation of the Kre\u{\i}n orthogonal complement of $\gr\phi$, gives
\[
\Gamma(\lV^{[\perp]})=(\gr\phi)^{[\perp]}
=\{(\phi^*\xi_0,\xi_0)^\top:\xi_0\in\lK_0\},
\]
which yields the second identity in~\eqref{eq:v-classif}. Conversely, a non-expansive $\phi$ makes $\gr\phi$ maximal non-negative and its orthogonal complement maximal non-positive; pulling these spaces back by $\Gamma$ gives the {(V)}-boundary conditions. Uniqueness follows from the uniqueness in Lemma~\ref{lem:graph-rep}.
\end{proof}

\begin{remark}\label{rem:v-equal-tildev}
Let $(\lV,\widetilde\lV)$ satisfy the {(V)}-boundary conditions. Then
\[
    \lV=\widetilde\lV
\]
if and only if the operator $\phi$ in Theorem~\ref{thm:v-classif} is unitary. In this case,
\[
    \lV=\ker(\phi\Gamma_1-\Gamma_0).
\]
\end{remark}

\begin{remark}
It was proved in \cite{BES25}, that the {(V)}-boundary conditions are equivalent to $m$-accretivity of the corresponding realisation. Hence, $m$-accretive realisations are parametrised by non-expansive operators.
\end{remark}

The following result identifies the {(V)}-boundary conditions to which a boundary quadruple or a boundary triple can be associated.

\begin{theorem}\label{thm:v=bt}
Let $(T_0,\widetilde{T}_0)$ be a joint pair of closed abstract Friedrichs operators on $\lH$, and let $(\lV,\widetilde\lV)$ satisfy the \emph{(V)}-boundary conditions. Then the following assertions hold:
\begin{itemize}
    \item[\rm{(i)}] If
    \begin{equation}\label{eq:bq2v}
        \lV+\widetilde\lV=\lW,
    \end{equation}
    then there exists a boundary quadruple $(\lK_1,\lK_0,\Gamma_1,\Gamma_0)$ such that
    \begin{equation}\label{eq:v-kernels-bq}
        \lV=\ker\Gamma_0,
        \qquad
        \widetilde\lV=\ker\Gamma_1.
    \end{equation}
    Conversely, for every boundary quadruple $(\lK_1,\lK_0,\Gamma_1,\Gamma_0)$, the pair $(\ker\Gamma_0,\ker\Gamma_1)$ satisfies the \emph{(V)}-boundary conditions.

    \item[\rm{(ii)}] If
    \[
        \lV=\widetilde\lV,
    \]
    then there exists a boundary triple $(\gamma_0,\gamma_1,\lX)$ such that
    \begin{equation}\label{eq:v-kernel-bt}
        \lV=\ker\gamma_0.
    \end{equation}
    Conversely, for every boundary triple $(\lX, \gamma_1, \gamma_0)$, both $(\ker\gamma_0,\ker\gamma_0)$ and $(\ker\gamma_1,\ker\gamma_1)$ satisfy the \emph{(V)}-boundary conditions.
\end{itemize}
\end{theorem}

\begin{proof}
The converse assertions follow from Corollary~\ref{cor:bqv} and Remark~\ref{rem:btv}, respectively.

\begin{itemize}
    \item[(i)] Let $q:\lW\to\lW/\lW_0$ be the quotient map, and put
    \[
        \lK_1:=q(\lV),
        \qquad
        \lK_0:=q(\widetilde\lV).
    \]
    Since $\lV\cap\widetilde\lV=\lW_0$ and \eqref{eq:bq2v} holds, and one has the direct decomposition
    \[
        \lW/\lW_0=\lK_1\dotplus\lK_0.
    \]
    Clearly, spaces $\lK_1$ and $\lK_0$ are positive- and negative-definite, respectively, and these subspaces are closed and mutually $[\perp]$-orthogonal. Since they form a closed fundamental decomposition of the Kre\u{\i}n space $\lW/\lW_0$, the forms
    \[
        \iscp{\cdot}{\cdot}
        \quad\text{on }\lK_1,
        \qquad
        -\iscp{\cdot}{\cdot}
        \quad\text{on }\lK_0
    \]
    are Hilbert inner products; cf.~\cite[Chapter V, p.~100]{Bo}.

    Let $P_1$ and $P_0$ be the projections associated with the above decomposition and define
    \[
        \Gamma_1:=P_1\circ q,
        \qquad
        \Gamma_0:=P_0\circ q.
    \]
    Then $\Gamma=(\Gamma_1,\Gamma_0)^\top$ becomes surjective and
    \[
        \iscp{\vu}{\vv}
        =\scp{\Gamma_1\vu}{\Gamma_1\vv}_{\lK_1}
        -\scp{\Gamma_0\vu}{\Gamma_0\vv}_{\lK_0},
        \qquad \vu,\vv\in\lW.
    \]
    Hence $(\lK_1,\lK_0,\Gamma_1,\Gamma_0)$ is a boundary quadruple. The identities in \eqref{eq:v-kernels-bq} follow immediately from the construction.

    \item[(ii)] Consider a boundary quadruple $(\lK_1,\lK_0,\Gamma_1,\Gamma_0)$. Assume $\lV=\widetilde\lV$, then by Theorem~\ref{thm:v-classif} and Remark~\ref{rem:v-equal-tildev}, there exists a unitary isomorphism $\phi:\lK_1\to\lK_0$ such that
    \[
        \lV=\ker(\phi\Gamma_1-\Gamma_0).
    \]
    Set $\lX:=\lK_0$ and define
    \[
        \gamma_0:=\frac{1}{\sqrt{2}}(\phi\Gamma_1-\Gamma_0),
        \qquad
        \gamma_1:=\frac{1}{\sqrt{2}}(\phi\Gamma_1+\Gamma_0).
    \]
    By Corollary~\ref{cor:bt}, $(\gamma_0,\gamma_1,\lX)$ is a boundary triple, and $\lV=\ker\gamma_0$.
\end{itemize}
\end{proof}

\begin{remark}
If $\lV=\widetilde\lV$, then one can also choose a boundary triple such that $\lV=\ker\gamma_1$. Indeed, with the notation from the proof of Theorem~\ref{thm:v=bt}(ii), define
\[
    \gamma_0:=\frac{1}{\sqrt{2}}(\phi\Gamma_1+\Gamma_0),
    \qquad
    \gamma_1:=\frac{1}{\sqrt{2}}(\phi\Gamma_1-\Gamma_0).
\]
\end{remark}

\subsection{(M)-boundary conditions}

The connection between the {(V)}- and {(M)}-boundary conditions has been studied in \cite{EGC,ABcpde}. We now express the multiplicity of the $M$-operators associated with a fixed {(V)}-boundary condition in terms of a boundary quadruple.

We first recall the following result from \cite[Theorem~8]{ABcpde}.

\begin{theorem}\label{thm:ABcpde-thm08}
Let $\lW$ be the graph space and let $D$ be the boundary operator.
\begin{itemize}
    \item[(i)] Suppose that $\lV$ satisfies the \emph{(V)}-boundary conditions. Then there exists a closed subspace $\lW_2\subseteq\lW^-$ such that
    \[
        \lW=\lV\dotplus\lW_2.
    \]
    If $\lW_2$ is such a subspace and
    \[
        \lW_1:=\lV\cap\lW_0^\perp,
    \]
    so that
    \[
        \lW=\lW_0\dotplus\lW_1\dotplus\lW_2,
    \]
    and $q_0,q_1,q_2$ are the corresponding projections, then
    \[
        M:=D(q_1-q_2)\in\mathcal{L}(\lW;\lW')
    \]
    satisfies the \emph{(M)}-boundary conditions and
    \[
        \lV=\ker(D-M).
    \]

    \item[(ii)] Conversely, let $M\in\mathcal{L}(\lW;\lW')$ satisfy the \emph{(M)}-boundary conditions and put
    \[
        \lV:=\ker(D-M).
    \]
    Then
    \[
        \lW_2:=\ker(D+M)\cap\lW_0^\perp
    \]
    is a closed subspace of $\lW^-$ and
    \[
        \lW=\lV\dotplus\lW_2.
    \]
\end{itemize}
\end{theorem}

Let us first prove the following general result.

\begin{lemma}\label{lem:char-m}
Let $(T_0,\widetilde{T}_0)$ be a joint pair of closed abstract Friedrichs operators on $\lH$, let $(\lK_1,\lK_0,\Gamma_1,\Gamma_0)$ be a boundary quadruple, and let
\[
    \phi\in\mathcal{L}(\lK_1,\lK_0),
    \qquad
    \psi\in\mathcal{L}(\lK_0,\lK_1).
\]
Then the following assertions hold:
\begin{itemize}
    \item[(i)] The operator $\I_{\lK_1}-\psi\phi$ is injective if and only if
    \begin{equation}\label{eq:char-m1}
        \lW_0
        =\ker(\phi\Gamma_1-\Gamma_0)
        \cap\ker(\Gamma_1-\psi\Gamma_0).
    \end{equation}

    \item[(ii)] The operator $\I_{\lK_1}-\psi\phi$ is surjective if and only if
    \begin{equation}\label{eq:char-m2}
        \lW
        =\ker(\phi\Gamma_1-\Gamma_0)
        +\ker(\Gamma_1-\psi\Gamma_0).
    \end{equation}
\end{itemize}
\end{lemma}

\begin{proof}
\begin{itemize}
    \item[(i)] Suppose that $\I_{\lK_1}-\psi\phi$ is injective and let
    \[
        \vu\in\ker(\phi\Gamma_1-\Gamma_0)
        \cap\ker(\Gamma_1-\psi\Gamma_0).
    \]
    Then
    \[
        (\I_{\lK_1}-\psi\phi)\Gamma_1\vu
        =\Gamma_1\vu-\psi\Gamma_0\vu=0.
    \]
    Hence $\Gamma_1\vu=0$, and then $\Gamma_0\vu=\phi\Gamma_1\vu=0$. Thus $\vu\in\lW_0$. The opposite inclusion is immediate.

    Conversely, suppose that \eqref{eq:char-m1} holds, and let $ \xi_1\in\ker(\I_{\lK_1}-\psi\phi)$.
    By the surjectivity of $\Gamma$, there exists $\vu\in\lW$ such that $\Gamma_1\vu=\xi_1$ and $\Gamma_0\vu=\phi\xi_1$.
    Then $\vu\in\ker(\phi\Gamma_1-\Gamma_0)$ and
    \[
        \Gamma_1\vu
        =\xi_1
        =\psi\phi\xi_1
        =\psi\Gamma_0\vu,
    \]
    so $\vu\in\ker(\Gamma_1-\psi\Gamma_0)$. By \eqref{eq:char-m1}, $\vu\in\lW_0$, and hence $\xi_1=\Gamma_1\vu=0$.

    \item[(ii)] Suppose that $\I_{\lK_1}-\psi\phi$ is surjective, and let $\vu\in\lW$. Set
    \[
        \xi_1:=\Gamma_1\vu,
        \qquad
        \xi_0:=\Gamma_0\vu.
    \]
    Choose $\mu_1\in\lK_1$ such that $(\I_{\lK_1}-\psi\phi)\mu_1
        =\xi_1-\psi\xi_0$,
    and put $\mu_0:=\xi_0-\phi\mu_1$.
    Then
    \[
        \mu_1+\psi\mu_0=\xi_1,
        \qquad
        \phi\mu_1+\mu_0=\xi_0.
    \]
    By the surjectivity of $\Gamma$, choose $\vu_1,\vu_2\in\lW$ such that $\Gamma_1\vu_1=\mu_1$, $\Gamma_0\vu_1=\phi\mu_1$, $\Gamma_1\vu_2=\psi\mu_0$ and $\Gamma_0\vu_2=\mu_0$.
    Thus,
    \[
        \vu_1\in\ker(\phi\Gamma_1-\Gamma_0),
        \qquad
        \vu_2\in\ker(\Gamma_1-\psi\Gamma_0).
    \]
    Moreover,
    \[
        \Gamma(\vu-\vu_1-\vu_2)=0,
    \]
    so $\vu-\vu_1-\vu_2\in\lW_0$. Since $\lW_0$ is contained in both kernels, \eqref{eq:char-m2} follows.

    Conversely, suppose that \eqref{eq:char-m2} holds, and let $\xi_1\in\lK_1$. Choose $\vu\in\lW$ such that $
        \Gamma_1\vu=\xi_1$ and $\Gamma_0\vu=0$.
    Write $\vu=\vu_1+\vu_2$, where
    \[
        \vu_1\in\ker(\phi\Gamma_1-\Gamma_0),
        \qquad
        \vu_2\in\ker(\Gamma_1-\psi\Gamma_0).
    \]
    Put
    \[
        \mu_1:=\Gamma_1\vu_1,
        \qquad
        \mu_0:=\Gamma_0\vu_2\,,
    \]
    then $\Gamma_0\vu_1=\phi\mu_1$ and $\Gamma_1\vu_2=\psi\mu_0$.
    Since $\Gamma_0\vu=0$, one has $\mu_0=-\phi\mu_1$, and therefore
    \[
        \xi_1
        =\Gamma_1\vu
        =\mu_1+\psi\mu_0
        =(\I_{\lK_1}-\psi\phi)\mu_1.
    \]
    Hence, $\I_{\lK_1}-\psi\phi$ is surjective.
\end{itemize}
\end{proof}

\begin{remark}\label{rem:schur}
The operator $\I_{\lK_1}-\psi\phi$ is bijective if and only if $\I_{\lK_0}-\phi\psi$ is bijective. In this case, with
\[
    S_1:=(\I_{\lK_1}-\psi\phi)^{-1},
    \qquad
    S_0:=(\I_{\lK_0}-\phi\psi)^{-1},
\]
one has
\begin{equation}\label{eq:schur}
    S_0=\I_{\lK_0}+\phi S_1\psi,
    \qquad
    S_1=\I_{\lK_1}+\psi S_0\phi.
\end{equation}
The operators $\I_{\lK_1}-\psi\phi$ and $\I_{\lK_0}-\phi\psi$ are the \emph{Schur complements} associated with
\[
    B:=
    \begin{pmatrix}
        \I_{\lK_1} & \psi\\
        \phi & \I_{\lK_0}
    \end{pmatrix};
\]
see \cite[Chapters~1 and~3]{Zhang05}. Moreover,
\begin{equation}\label{eq:B-inverse}
    B^{-1}
    =
    \begin{pmatrix}
        S_1 & -S_1\psi\\
        -S_0\phi & S_0
    \end{pmatrix}.
\end{equation}
\end{remark}

We can now parametrise all negative complements occurring in Theorem~\ref{thm:ABcpde-thm08}.

\begin{theorem}\label{thm:m-psi}
Let
\[
    \lV=\ker(\phi\Gamma_1-\Gamma_0),
\]
where $\phi:\lK_1\to\lK_0$ is non-expansive, and let $\lW_2\subseteq\lW$ be closed. Then the following assertions are equivalent:
\begin{itemize}
    \item[\emph{(i)}] $\lW_2\subseteq\lW^-$ and
    \begin{equation}\label{eq:sum-m-op}
        \lW=\lV\dotplus\lW_2;
    \end{equation}

    \item[\emph{(ii)}] there exists a non-expansive operator $\psi:\lK_0\to\lK_1$ such that $\I_{\lK_1}-\psi\phi$ is bijective and
    \begin{equation}\label{eq:W2-psi}
        \lW_0\dotplus\lW_2
        =\ker(\Gamma_1-\psi\Gamma_0).
    \end{equation}
\end{itemize}
\end{theorem}

\begin{proof}
Assume first that {\rm (i)} holds, and set
\[
    \lV':=\lW_0\dotplus\lW_2.
\]
Since $\lW=\lV\dotplus\lW_2$ and $\lW_0\subseteq\lV$, the subspace
$\lV'$ is closed. Moreover,
\[
    \lV\cap\lV'=\lW_0,
    \qquad
    \lV+\lV'=\lW.
\]
Set
\[
    \Phi:=\Gamma(\lV)=\gr\phi,
    \qquad
    \Psi:=\Gamma(\lV').
\]
By Proposition~\ref{prop:relation-param}, $\Psi$ is closed, and the
preceding identities imply
\[
    \lK=\Phi\dotplus\Psi.
\]
Furthermore, $\Psi$ is non-positive because
$\lV'\subseteq\lW^-$. Since $\phi$ is non-expansive,
Lemma~\ref{lem:graph-rep}(i) shows that $\Phi$ is maximal
non-negative. Hence Lemma~\ref{lem:complement} implies that $\Psi$ is
maximal non-positive.

By Lemma~\ref{lem:graph-rep}(ii), there exists a unique
non-expansive operator $\psi:\lK_0\to\lK_1$ such that
\[
    \Psi
    =
    \left\{
        \begin{pmatrix}
            \psi\xi_0\\
            \xi_0
        \end{pmatrix}
        :
        \xi_0\in\lK_0
    \right\}.
\]
Since $\lV'$ contains $\ker\Gamma=\lW_0$, it follows that \[\lV'
    =
    \ker(\Gamma_1-\psi\Gamma_0)\;.\]
Finally, the conditions \eqref{eq:char-m1} and \eqref{eq:char-m2} of Lemma~\ref{lem:char-m} are satisfied, implying
that $\I_{\lK_1}-\psi\phi$ is bijective. Thus {\rm (ii)} holds.

Conversely, assume {\rm (ii)} and put
\[
    \lV'
    :=
    \ker(\Gamma_1-\psi\Gamma_0)
    =
    \lW_0\dotplus\lW_2.
\]
Since $\psi$ is non-expansive,
\[
    \lV'\subseteq\lW^-,
\]
and hence $\lW_2\subseteq\lW^-$. By Lemma~\ref{lem:char-m}, the
bijectivity of $\I_{\lK_1}-\psi\phi$ gives
\[
    \lV\cap\lV'=\lW_0,
    \qquad
    \lV+\lV'=\lW.
\]
Since $\lW_0\subseteq\lV$ and
$\lV'=\lW_0\dotplus\lW_2$, these identities reduce to
\[
    \lW=\lV\dotplus\lW_2.
\]
Therefore {\rm (i)} holds.
   
\end{proof}

\begin{remark}\label{rem:canonical-complement-scope}
For a fixed $\psi$, equation \eqref{eq:W2-psi} determines the enlarged negative subspace $\lW_0\dotplus\lW_2$, but not necessarily a unique complement of $\lW_0$ inside that subspace. In the construction below we choose the canonical Hilbert-space complement
\[
    \lW_{2,\psi}^{\,0}
    :=\ker(\Gamma_1-\psi\Gamma_0)\cap\lW_0^\perp.
\]
Accordingly, the resulting operators form a canonical family associated with the admissible parameters $\psi$.
\end{remark}

\begin{remark}
A canonical admissible choice is $\psi=-\phi^*$. Indeed, $\psi$ is non-expansive and
\[
    \I_{\lK_1}-\psi\phi
    =\I_{\lK_1}+\phi^*\phi
\]
is bijective. Thus,
\[
    \lW_0\dotplus\lW_2
    =\ker(\Gamma_1+\phi^*\Gamma_0)
\]
defines one of the complements occurring in Theorem~\ref{thm:ABcpde-thm08}.
\end{remark}

For the explicit representation of the associated $M$-operator, let
\[
    \Phi:=\left\{(\xi_1,\phi\xi_1)^\top:\xi_1\in\lK_1\right\},
    \qquad
    \Psi:=\left\{(\psi\xi_0,\xi_0)^\top:\xi_0\in\lK_0\right\}.
\]
Let $\lM:=\lW_0^\perp$, where the orthogonal complement is taken in the graph space $\lW$, and let
\[
    \widetilde\Gamma:=\Gamma|_{\lM}:\lM\to\lK.
\]
Then $\widetilde\Gamma$ is a topological isomorphism. Its inverse satisfies
\begin{equation}\label{eq:Gamma-right-inverse}
    \Gamma\widetilde\Gamma^{-1}=\I_{\lK},
    \qquad
    \widetilde\Gamma^{-1}\Gamma=\mathcal{P}_{\lM},
\end{equation}
where $\mathcal{P}_{\lM}$ is the Hilbert-space orthogonal projection of $\lW$ onto $\lM$.

Let
\begin{equation}\label{eq:B-def}
    B:=
    \begin{pmatrix}
        \I_{\lK_1} & \psi\\
        \phi & \I_{\lK_0}
    \end{pmatrix}.
\end{equation}
Under the equivalent conditions of Theorem~\ref{thm:m-psi}, $B$ is bijective and
\begin{equation}\label{eq:decomposition}
    \lW
    =\lW_0
    \dotplus\widetilde\Gamma^{-1}(\Phi)
    \dotplus\widetilde\Gamma^{-1}(\Psi).
\end{equation}
The projections associated with this decomposition are
\begin{equation}\label{eq:projectors}
    p_1
    =\widetilde\Gamma^{-1}
    \begin{pmatrix}
        \I_{\lK_1} & 0\\
        \phi & 0
    \end{pmatrix}
    B^{-1}\Gamma,
    \qquad
    p_2
    =\widetilde\Gamma^{-1}
    \begin{pmatrix}
        0 & \psi\\
        0 & \I_{\lK_0}
    \end{pmatrix}
    B^{-1}\Gamma.
\end{equation}
The boundary operator has the representation
\begin{equation}\label{eq:rep-d}
    D=\Gamma'\J_{\lK}\Gamma\,,
\end{equation}
where $\Gamma':\lK\to \lW'$ is the dual-adjoint defined as
\begin{align*}
    \dup{\lW'}{\Gamma'\mxi}{\vu}{\lW}=\scp{\mxi}{\Gamma\vu}_\lK
\end{align*}
Therefore, the operator $M_\psi:=D(p_1-p_2)$ can be written as
\begin{equation}\label{eq:rep-M}
    M_\psi
    =\Gamma'
    \begin{pmatrix}
        \I_{\lK_1} & -\psi\\
        -\phi & \I_{\lK_0}
    \end{pmatrix}
    B^{-1}\Gamma.
\end{equation}

We first verify the projection formulas in \eqref{eq:projectors}.

\begin{lemma}\label{lem:mpsi}
Let $p_1$ and $p_2$ be defined by \eqref{eq:projectors}. Then
\begin{itemize}
    \item[(i)] $p_1^2=p_1$ and $p_2^2=p_2$;
    \item[(ii)] $p_1p_2=p_2p_1=0$;
    \item[(iii)] $D(p_1+p_2)=D$.
\end{itemize}
\end{lemma}

\begin{proof}
Set
\[
    B_1:=
    \begin{pmatrix}
        \I_{\lK_1} & 0\\
        \phi & 0
    \end{pmatrix},
    \qquad
    B_2:=
    \begin{pmatrix}
        0 & \psi\\
        0 & \I_{\lK_0}
    \end{pmatrix},
\]
and put
\[
    P_1:=B_1B^{-1},
    \qquad
    P_2:=B_2B^{-1}.
\]
For $\mxi\in\lK$, write
\[
    B^{-1}\mxi=(\mu_1,\mu_0)^\top.
\]
Then
\[
    P_1\mxi=(\mu_1,\phi\mu_1)^\top\in\Phi,
    \qquad
    P_2\mxi=(\psi\mu_0,\mu_0)^\top\in\Psi,
\]
and
\[
    \mxi=P_1\mxi+P_2\mxi.
\]
Thus $P_1$ and $P_2$ are the complementary projections of $\lK=\Phi\dotplus\Psi$. Consequently,
\[
    P_1^2=P_1,
    \qquad
    P_2^2=P_2,
    \qquad
    P_1P_2=P_2P_1=0,
    \qquad
    P_1+P_2=\I_{\lK}.
\]
Using \eqref{eq:Gamma-right-inverse}, we obtain
\[
    p_ip_j
    =\widetilde\Gamma^{-1}P_iP_j\Gamma,
    \qquad i,j\in\{1,2\},
\]
which proves \rm(i) and \rm(ii). Finally,
\begin{align*}
    D(p_1+p_2)
    &=\Gamma'\J_{\lK}\Gamma
      \widetilde\Gamma^{-1}(P_1+P_2)\Gamma\\
    &=\Gamma'\J_{\lK}\Gamma
    =D.
\end{align*}
\end{proof}

By Theorem \ref{thm:ABcpde-thm08}, the operator $M_{\psi}$ is constructed so that it is indeed an $M-$operator with $\ker(D-M_\psi)=\ker(\phi \Gamma_1-\Gamma_0)$ and $\ker(D+M_\psi)=\ker(\Gamma_1-\psi\Gamma_0)$. 
Moreover, the projections $p_1$ and $p_2$ fit into the framework of Theorem \ref{thm:ABcpde-thm08}. Hence, we refer to the proof of Theorem \ref{thm:ABcpde-thm08}, in \cite[Theorem 8]{ABcpde} to conclude the following result.
\begin{theorem}\label{thm:m-op}
Let $\phi:\lK_1\to\lK_0$ and $\psi:\lK_0\to\lK_1$ be non-expansive, and assume that $\I_{\lK_1}-\psi\phi$ is bijective. Then the operator $M_\psi$ in \eqref{eq:rep-M} satisfies the \emph{(M)}-boundary conditions, and
\[
    \ker(D-M_\psi)=\ker(\phi\Gamma_1-\Gamma_0),\qquad
    \ker(D+M_\psi)=\ker(\Gamma_1-\psi\Gamma_0).
\]
\end{theorem}

\begin{remark}
Two useful choices are $\psi=0$ and $\psi=-\phi^*$. For $\psi=0$, one has
\[
    \lW=\dom T_\phi+\ker\Gamma_1,
\]
and \eqref{eq:rep-M} reduces to
\begin{equation}\label{eq:M-psi-zero}
    \dup{\lW'}{M_0\vu}{\vv}{\lW}
    =\scp{\Gamma_1\vu}{\Gamma_1\vv}_{\lK_1}
    +\scp{\Gamma_0\vu}{\Gamma_0\vv}_{\lK_0}
    -2\scp{\phi\Gamma_1\vu}{\Gamma_0\vv}_{\lK_0}.
\end{equation}
For $\psi=-\phi^*$, the required bijectivity follows from the invertibility of $\I_{\lK_1}+\phi^*\phi$.
\end{remark}

\section{Examples}\label{sec:examples}

We first illustrate the theory on a first-order ordinary differential operator.

\begin{example}\label{ex:first-order-ode}
Let
\[
    \lH=L^2(a,b;\C),
    \qquad
    \lD=C_c^\infty(a,b),
\]
and let $\beta\in L^\infty(a,b;\R)$ satisfy $\beta\geq\mu>0$ almost everywhere on $(a,b)$. Define
\begin{equation}\label{eq:ode-pair}
    Tu=u'+\beta u,
    \qquad
    \widetilde Tu=-u'+\beta u,
    \qquad u\in\lD.
\end{equation}
Then $(T_0,\widetilde T_0)$ is a joint pair of closed abstract Friedrichs operators for $T_0=\overline T$ and $\widetilde T_0=\overline{\widetilde T}$. Its graph space and minimal space are
\[
    \lW=H^1(a,b),
    \qquad
    \lW_0=H_0^1(a,b),
\]
and the boundary form is
\begin{equation}\label{eq:ode-boundary-form}
    \iscp{u}{v}
    =u(b)\overline{v(b)}-u(a)\overline{v(a)},
    \qquad u,v\in\lW.
\end{equation}
For direct analyses of this example, we refer to \cite{AEM-2017,ES22}; the von Neumann parametrisation is discussed in \cite[Example 3.6]{ES23}, and the corresponding {(M)}-boundary conditions are considered in \cite[Example 5.6]{BES25}.

\medskip
\noindent\textbf{Boundary quadruple.}
Set
\[
    \lK_1=\lK_0=\C,
    \qquad
    \Gamma_1u=u(b),
    \qquad
    \Gamma_0u=u(a).
\]
Then $(\lK_1,\lK_0,\Gamma_1,\Gamma_0)$ is a boundary quadruple. For $\alpha\in\C$, let
\[
    \phi_\alpha\xi=\alpha\xi,
    \qquad
    \lV_\alpha
    :=\ker(\phi_\alpha\Gamma_1-\Gamma_0)
    =\{u\in H^1(a,b):u(a)=\alpha u(b)\}.
\]
Put
\[
    c_\beta:=\exp\left(-\int_a^b\beta(s)\,ds\right).
\]
Since $\beta\geq\mu>0$, one has $0<c_\beta<1$. If $\vv\in\ker T_1$, then
\[
    \vv(x)=\vv(a)\exp\left(-\int_a^x\beta(s)\,ds\right),
\]
and hence the zero-reference operator is
\[
    Q_0=c_\beta\I_\C.
\]
Therefore, Theorem~\ref{thm:classification} gives
\begin{itemize}
    \item $T_1|_{\lV_\alpha}$ is bijective if and only if
    \[
        1-\alpha c_\beta\neq0;
    \]

    \item $(\lV_\alpha,\lV_\alpha^{[\perp]})$ satisfies the {(V)}-boundary conditions if and only if
    \[
        |\alpha|\leq1;
    \]

    \item $\lV_\alpha=\lV_\alpha^{[\perp]}$ if and only if
    \[
        |\alpha|=1.
    \]
\end{itemize}
In particular, every non-expansive parameter gives a bijective realisation because $c_\beta<1$.

\medskip
\noindent\textbf{{(M)}-boundary conditions.}
Fix $\alpha,\tau\in[-1,1]$ such that $1-\alpha\tau\neq0$, and identify $\phi$ and $\psi$ with multiplication by $\alpha$ and $\tau$, respectively. The matrix in Theorem~\ref{thm:m-op} becomes
\begin{equation}\label{eq:ode-M-matrix}
    C_{\alpha,\tau}
    :=
    \begin{pmatrix}
        1 & -\tau\\
        -\alpha & 1
    \end{pmatrix}
    \begin{pmatrix}
        1 & \tau\\
        \alpha & 1
    \end{pmatrix}^{-1}
    =
    \frac{1}{1-\alpha\tau}
    \begin{pmatrix}
        1+\alpha\tau & -2\tau\\
        -2\alpha & 1+\alpha\tau
    \end{pmatrix}.
\end{equation}
Consequently,
\begin{align}\label{eq:ode-M-form}
    \dup{\lW'}{M_\tau u}{v}{\lW}
    =\frac{1}{1-\alpha\tau}\bigl(&
        (1+\alpha\tau)u(b)\overline{v(b)}
        -2\tau u(a)\overline{v(b)}
        \notag\\
        &-2\alpha u(b)\overline{v(a)}
        +(1+\alpha\tau)u(a)\overline{v(a)}
    \bigr)
\end{align}
for all $u,v\in\lW$. For $\tau=0$, this reduces to
\begin{equation}\label{eq:ode-M-zero}
    \dup{\lW'}{M_0u}{v}{\lW}
    =u(b)\overline{v(b)}+u(a)\overline{v(a)}
    -2\alpha u(b)\overline{v(a)}.
\end{equation}
\end{example}

\begin{example}[stationary diffusion equation]\label{ex:stationary-diffusion}
Let $\Omega\subset\R^d$ be an open and bounded set with Lipschitz boundary
$\partial\Omega$. Consider the second-order equation
\begin{equation*}
    -\Delta u+u=f.
\end{equation*}
It can be written as the first-order system
\begin{equation}\label{eq:SDE-system}
    \left\{
    \begin{array}{ll}
      \mp=-\nabla u,\\[1mm]
      \operatorname{div}\mp+u=f.
    \end{array}
    \right.
\end{equation}
Let
\[
    \mA_k=\ve_k\otimes\ve_{d+1}+\ve_{d+1}\otimes\ve_k
    \in \mathrm{M}_{d+1}(\R),
    \qquad k=1,\ldots,d,
\]
where $(\ve_1,\ldots,\ve_{d+1})$ is the standard basis of $\R^{d+1}$,
and let $\mB$ be the identity matrix of order $d+1$. On $C_c^\infty(\Omega;\C^{d+1})$, define the preminimal joint pair
\begin{align*}
    T\vu&:=\sum_{k=1}^d\partial_k(\mA_k\vu)+\mB\vu,\\
    \widetilde T\vu&:=-\sum_{k=1}^d\partial_k(\mA_k\vu)+\mB\vu,
\end{align*}
where
\[
    \vu=\begin{bmatrix}\mp\\ u\end{bmatrix},
    \qquad
    \vf=\begin{bmatrix}0\\ f\end{bmatrix}.
\]
Let $T_0=\overline{T}$ and $\widetilde T_0=\overline{\widetilde T}$. Then $(T_0,\widetilde T_0)$ is a joint pair of closed abstract Friedrichs operators; see~\cite[Section~6]{ABjde}. The first-order system \eqref{eq:SDE-system} is represented on the graph space by
\[
    T_1\vu=\vf,
\]
and its closed realisations are restrictions of the maximal operator $T_1$.
We use the standard notation in which $T_0$ denotes the minimal operator
and $T_1$ the corresponding maximal operator. The graph space, the minimal
space and the boundary form below are described in \cite[Section~6]{ABjde}:
\begin{align*}
    \lW
    &:=L^2_{\operatorname{div}}(\Omega;\C^d)\times H^1(\Omega;\C),\\
    \lW_0
    &:=L^2_{\operatorname{div},0}(\Omega;\C^d)\times H^1_0(\Omega;\C).
\end{align*}
Here
\[
    L^2_{\operatorname{div}}(\Omega;\C^d)
    :=\bigl\{\mp\in L^2(\Omega;\C^d):
    \operatorname{div}\mp\in L^2(\Omega;\C)\bigr\},
\]
and $L^2_{\operatorname{div},0}(\Omega;\C^d)$ is the closure of
$C_c^\infty(\Omega;\C^d)$ in the graph norm of
$L^2_{\operatorname{div}}(\Omega;\C^d)$.

Let
\[
    \mT_0:H^1(\Omega;\C)\to H^{1/2}(\partial\Omega;\C)
\]
be the Dirichlet trace and
\[
    \mT_{\mnu}:L^2_{\operatorname{div}}(\Omega;\C^d)
    \to H^{-1/2}(\partial\Omega;\C)
\]
be the normal trace. The Dirichlet trace theorem is given in \cite[Theorem~3.38]{McLean}, while the normal trace and its Green formula are given in \cite[Chapter~I, Theorem~2.5]{GiraultRaviart}. Set
\[
    \lX:=H^{1/2}(\partial\Omega;\C),
\]
and let
\[
    R:H^{-1/2}(\partial\Omega;\C)\to\lX
\]
be the Riesz map, which is a unitary isomorphism, characterised by
\[
    \scp{R\eta}{\xi}_{\lX}
    =\dup{-1/2}{\eta}{\xi}{1/2},
    \qquad
    \eta\in H^{-1/2}(\partial\Omega;\C),
    \quad \xi\in\lX.
\]
For
\[
    \vu=\begin{pmatrix}\mp\\u\end{pmatrix},
    \qquad
    \vv=\begin{pmatrix}\mq\\v\end{pmatrix}
    \in\lW,
\]
the boundary form is
\begin{equation}\label{eq:SDE-boundary}
    \iscp{\vu}{\vv}
    =\scp{R\mT_{\mnu}\mp}{\mT_0v}_{\lX}
    +\scp{\mT_0u}{R\mT_{\mnu}\mq}_{\lX}.
\end{equation}
Equivalently, if $D\in\mathcal L(\lW;\lW')$ is the boundary operator, then
\[
    \dup{\lW'}{D\vu}{\vv}{\lW}=\iscp{\vu}{\vv}.
\]

\medskip
\noindent\textbf{Boundary quadruple.}
Let $\lK_1=\lK_0=\lX$ and define
\begin{equation}\label{eq:SDE-Gamma}
    \Gamma_1\begin{pmatrix}\mp\\u\end{pmatrix}
    :=\frac12\mT_0u+R\mT_{\mnu}\mp,
    \qquad
    \Gamma_0\begin{pmatrix}\mp\\u\end{pmatrix}
    :=\frac12\mT_0u-R\mT_{\mnu}\mp.
\end{equation}
Then $(\lK_1,\lK_0,\Gamma_1,\Gamma_0)$ is a boundary quadruple. Indeed,
expansion gives
\[
    \iscp{\vu}{\vv}
    =\scp{\Gamma_1\vu}{\Gamma_1\vv}_{\lX}
     -\scp{\Gamma_0\vu}{\Gamma_0\vv}_{\lX},
\]
which is precisely \eqref{eq:SDE-boundary}. Moreover, the Dirichlet and
normal trace maps are surjective on bounded Lipschitz domains. Since the
linear transformation
\[
    (d,n)\longmapsto\left(\frac12d+n,\frac12d-n\right)
\]
is invertible on $\lX\oplus\lX$, the map
$\Gamma=(\Gamma_1,\Gamma_0)^\top:\lW\to\lX\oplus\lX$ is surjective.

\medskip
\noindent\textbf{The kernel of $T_1$ and the reference operator $Q_0$.}
Following \cite[Example~5.7]{BES25}, the kernel of the maximal operator is
\begin{equation}\label{eq:SDE-ker-T1}
    \ker T_1
    =\left\{
       \begin{pmatrix}\mp\\u\end{pmatrix}\in\lW:
       \mp=-\nabla u,\quad u=-\operatorname{div}\mp
      \right\}.
\end{equation}
Equivalently,
\[
    \begin{pmatrix}\mp\\u\end{pmatrix}\in\ker T_1
    \quad\Longleftrightarrow\quad
    -\Delta u+u=0\ \text{in }\Omega,
    \qquad \mp=-\nabla u.
\]

Define
\[
    \tau_D:\ker T_1\to\lX,
    \qquad
    \tau_D\begin{pmatrix}\mp\\u\end{pmatrix}:=\mT_0u.
\]
For $g\in\lX$, let $w_g\in H^1(\Omega;\C)$ be the unique weak solution of
\begin{equation}\label{eq:SDE-Dirichlet-problem}
    -\Delta w_g+w_g=0\quad\text{in }\Omega,
    \qquad \mT_0w_g=g,
\end{equation}
and set
\begin{equation}\label{eq:SDE-ED}
    E_Dg:=\begin{pmatrix}-\nabla w_g\\w_g\end{pmatrix}.
\end{equation}
The existence of a bounded right inverse of the Dirichlet trace follows
from \cite[Theorem~3.38]{McLean}, and the existence, uniqueness and
continuous dependence of $w_g$ then follow from the Lax--Milgram theorem
\cite[Corollary~5.8]{Brezis}. Hence $E_D:\lX\to\ker T_1$ is bounded and
\[
    \tau_DE_D=\I_{\lX},
    \qquad
    E_D\tau_D=\I_{\ker T_1}.
\]
Thus $\tau_D$ is a topological isomorphism from $\ker T_1$ onto $\lX$,
with inverse $E_D$.

Let
\begin{equation}\label{eq:SDE-DtN}
    \Lambda:\lX\to H^{-1/2}(\partial\Omega;\C),
    \qquad
    \Lambda g:=\mT_{\mnu}(\nabla w_g),
\end{equation}
be the Dirichlet-to-Neumann operator for $-\Delta+1$, and set
\begin{equation}\label{eq:SDE-A}
    \mathcal A:=R\Lambda\in\mathcal L(\lX).
\end{equation}
The boundedness of $\mathcal A$ follows from the boundedness of $E_D$ and
of the normal trace. Since the vector component of $E_Dg$ is
$-\nabla w_g$, we have
\[
    R\mT_{\mnu}(-\nabla w_g)=-\mathcal Ag.
\]
Therefore, by \eqref{eq:SDE-Gamma},
\begin{align}
    \Gamma_0E_Dg
    &=\left(\frac12\I_{\lX}+\mathcal A\right)g,
    \label{eq:SDE-Gamma0-ED}\\
    \Gamma_1E_Dg
    &=\left(\frac12\I_{\lX}-\mathcal A\right)g.
    \label{eq:SDE-Gamma1-ED}
\end{align}

Set
\[
    B_0:=\Gamma_0|_{\ker T_1}:\ker T_1\to\lX.
\]
By Proposition~4.3, $B_0$ is a topological isomorphism and the reference
operator
\[
    Q_0:=\Gamma_1B_0^{-1}
\]
is a strict contraction. On the other hand,
\eqref{eq:SDE-Gamma0-ED} gives
\begin{equation}\label{eq:SDE-B0-ED}
    B_0E_D=\frac12\I_{\lX}+\mathcal A.
\end{equation}
Since $B_0$ and $E_D$ are topological isomorphisms, it follows directly that
$\frac12\I_{\lX}+\mathcal A$ is a topological isomorphism and
\begin{equation}\label{eq:SDE-B0-inverse}
    B_0^{-1}
    =E_D\left(\frac12\I_{\lX}+\mathcal A\right)^{-1}.
\end{equation}
Consequently,
\begin{align}
    Q_0
    &=\Gamma_1B_0^{-1}\notag\\
    &=\Gamma_1E_D
      \left(\frac12\I_{\lX}+\mathcal A\right)^{-1}\notag\\
    &=\left(\frac12\I_{\lX}-\mathcal A\right)
      \left(\frac12\I_{\lX}+\mathcal A\right)^{-1}.
    \label{eq:SDE-Q0-DtN}
\end{align}
Thus the concrete formula for $Q_0$ follows solely by identifying the two
boundary components on $\ker T_1$. Moreover, Proposition~4.3 gives
\begin{equation}\label{eq:SDE-Q0-strict}
    \|Q_0\|_{\mathcal L(\lX)}<1.
\end{equation}
No additional spectral property of the Dirichlet-to-Neumann operator is
needed.

For $\phi\in\mathcal L(\lX)$, define
$T_\phi:=T_1|_{\dom T_\phi}$, where
\begin{equation}\label{eq:SDE-domain}
    \dom T_\phi
    :=\left\{
        \begin{pmatrix}\mp\\u\end{pmatrix}\in\lW:
        \phi\left(\frac12\mT_0u+R\mT_{\mnu}\mp\right)
        =\frac12\mT_0u-R\mT_{\mnu}\mp
      \right\}.
\end{equation}
Theorems~4.4 and~4.7 yield:
\begin{itemize}
    \item $T_\phi$ is bijective if and only if
    $\I_{\lX}-\phi Q_0$ is bijective;

    \item $(\dom T_\phi,\dom T_\phi^*)$ satisfies the \emph{(V)}-boundary
    conditions if and only if $\phi$ is non-expansive;

    \item $\dom T_\phi=\dom T_\phi^*$ if and only if $\phi$ is unitary.
\end{itemize}
In particular, if $\phi$ is non-expansive, then
\[
    \|\phi Q_0\|\leq\|Q_0\|<1,
\]
so $\I_{\lX}-\phi Q_0$ is bijective by the Neumann series. Hence every
non-expansive parameter gives a bijective realisation.

The standard boundary conditions are recovered as follows:
\begin{itemize}
    \item $\phi=\I_{\lX}$ gives $\mT_{\mnu}\mp=0$, the Neumann condition;

    \item $\phi=-\I_{\lX}$ gives $\mT_0u=0$, the homogeneous Dirichlet
    condition;

    \item if $\rho\in(-1,1)$ and $\phi=\rho\I_{\lX}$, then
    \[
        R\mT_{\mnu}\mp
        =\kappa_\rho\mT_0u,
        \qquad
        \kappa_\rho:=\frac{1-\rho}{2(1+\rho)}>0.
    \]
    This is a Robin-type condition in the Riesz boundary coordinates.
    The endpoint values $\rho=-1$ and $\rho=1$ give the Dirichlet and
    Neumann conditions, respectively.
\end{itemize}

\medskip
\noindent\emph{Relation with the classical Robin condition and the scope of the result:}
Let
\[
    J_{\partial}:\lX\to H^{-1/2}(\partial\Omega;\C)
\]
be the canonical embedding induced by the $L^2(\partial\Omega)$ pairing,
that is,
\[
    \dup{-1/2}{J_{\partial}g}{h}{1/2}
    =\int_{\partial\Omega}g\,\overline h\,dS,
    \qquad g,h\in\lX,
\]
and set
\[
    \mathcal C:=RJ_{\partial}\in\mathcal L(\lX).
\]
For $g,h\in\lX$, the definitions of $R$ and $J_{\partial}$ give
\[
    \scp{\mathcal Cg}{h}_{\lX}
    =
    \int_{\partial\Omega}g\,\overline h\,dS.
\]
Hence $\mathcal C$ is self-adjoint and
\[
    \scp{\mathcal Cg}{g}_{\lX}
    =
    \int_{\partial\Omega}|g|^2\,dS
    \geq0.
\] The positive classical
Robin conditions considered in \cite{ABcpde,ABjde} can be written as
\begin{equation}\label{eq:SDE-classical-Robin}
    \mT_{\mnu}\mp=\beta J_{\partial}\mT_0u,
    \qquad \beta>0,
\end{equation}
while $\beta=0$ gives the Neumann condition. In the present boundary
coordinates, \eqref{eq:SDE-classical-Robin} is equivalent to
\[
    \Gamma_1\vu
    =\left(\frac12\I_{\lX}+\beta\mathcal C\right)\mT_0u,
    \qquad
    \Gamma_0\vu
    =\left(\frac12\I_{\lX}-\beta\mathcal C\right)\mT_0u.
\]
For $\beta\geq0$, the operator
\[
    D_\beta:=\frac12\I_{\lX}+\beta\mathcal C
\]
is a topological isomorphism. Indeed, since $\mathcal C$ is self-adjoint and
non-negative,
\[
    \scp{D_\beta g}{g}_{\lX}
    =\frac12\|g\|_{\lX}^2
     +\beta\scp{\mathcal Cg}{g}_{\lX}
    \geq\frac12\|g\|_{\lX}^2.
\]
Hence $\|D_\beta g\|_{\lX}\geq\frac12\|g\|_{\lX}$, so $D_\beta$ is
injective and has closed range. Since $D_\beta$ is self-adjoint,
\[
    (\ran D_\beta)^\perp=\ker D_\beta=\{0\}.
\]
Thus its range is also dense and consequently equals $\lX$. Hence the
classical Robin domain is represented by
\begin{equation}\label{eq:SDE-classical-Robin-parameter}
    \Gamma_0\vu=\phi_\beta\Gamma_1\vu,
    \qquad
    \phi_\beta
    :=\left(\frac12\I_{\lX}-\beta\mathcal C\right)
      \left(\frac12\I_{\lX}+\beta\mathcal C\right)^{-1}.
\end{equation}
Moreover, if
$y=(\frac12\I_{\lX}+\beta\mathcal C)g$, then
\[
    \|y\|_{\lX}^2-\|\phi_\beta y\|_{\lX}^2
    =2\beta\scp{\mathcal Cg}{g}_{\lX}\geq0.
\]
Thus $\phi_\beta$ is non-expansive for $\beta\geq0$. Consequently, the
positive Robin conditions treated in \cite{ABcpde,ABjde} are recovered
from the non-expansive part of the present theory. Notice that the scalar
choice $\phi=\rho\I_{\lX}$ gives a Robin-type condition in the Riesz
boundary coordinates; it is not, in general, the same as the classical
scalar Robin condition \eqref{eq:SDE-classical-Robin} because
$R^{-1}$ and $J_{\partial}$ are different boundary maps.

The exact criterion also gives the corresponding solvability condition in
Dirichlet-to-Neumann coordinates. Put
\[
    D_\beta:=\frac12\I_{\lX}+\beta\mathcal C,
    \qquad
    N_\beta:=\frac12\I_{\lX}-\beta\mathcal C.
\]
Whenever $D_\beta$ is bijective, $\phi_\beta=N_\beta D_\beta^{-1}$, and a
direct calculation using \eqref{eq:SDE-Q0-DtN} gives
\begin{equation}\label{eq:SDE-classical-Robin-factor}
    D_\beta
    \left(\I_{\lX}-\phi_\beta Q_0\right)
    \left(\frac12\I_{\lX}+\mathcal A\right)
    =\mathcal A+\beta\mathcal C.
\end{equation}
No commutativity between $\mathcal A$ and $\mathcal C$ is used in this
identity. Since the two outer factors on the left are bijective,
Theorem~4.4 yields
\begin{equation}\label{eq:SDE-classical-Robin-bijective}
    T_{\phi_\beta}\text{ is bijective}
    \quad\Longleftrightarrow\quad
    \mathcal A+\beta\mathcal C\text{ is bijective}.
\end{equation}
For $\beta\geq0$, this bijectivity is already guaranteed by the
non-expansiveness of $\phi_\beta$. For $\beta<0$, however, $D_\beta$ may
fail to be bijective. In that case the Robin condition is still represented
by the closed boundary relation
\begin{equation}\label{eq:SDE-classical-Robin-relation}
    \Theta_\beta
    :=\left\{
       \begin{pmatrix}D_\beta d\\ N_\beta d\end{pmatrix}:
       d\in\lX
      \right\}\subseteq\lX\oplus\lX,
\end{equation}
This relation is closed. Indeed, $D_\beta+N_\beta=\I_{\lX}$. Hence, if
\[
    (D_\beta d_n,N_\beta d_n)\longrightarrow(\xi_1,\xi_0)
    \quad\text{in }\lX\oplus\lX,
\]
then
\[
    d_n=D_\beta d_n+N_\beta d_n
    \longrightarrow\xi_1+\xi_0.
\]
By continuity of $D_\beta$ and $N_\beta$, it follows that
\[
    (\xi_1,\xi_0)
    =\bigl(D_\beta(\xi_1+\xi_0),N_\beta(\xi_1+\xi_0)\bigr)
    \in\Theta_\beta.
\]
Therefore, the relation criterion of Theorem~4.1 applies directly.

This identifies precisely the additional conclusion of the present
framework. The positive classical Robin conditions from
\cite{ABcpde,ABjde} form a non-expansive special case. The new information
is the exact bijectivity criterion in the prescribed boundary coordinates
for arbitrary bounded graph parameters, including expansive and
operator-valued parameters, together with the relation criterion when the
boundary condition is not the graph of an everywhere-defined operator.
The novelty is therefore not the introduction of operator-valued Robin
conditions themselves, but their treatment within the boundary-quadruple
classification of abstract Friedrichs operators by the single transfer
operator $Q_0$.

For completeness, the scalar Riesz-coordinate family may be written
explicitly. Let $\phi=\rho\I_{\lX}$ with $\rho\in\R$. From
\eqref{eq:SDE-Q0-DtN},
\begin{align}
    \I_{\lX}-\rho Q_0
    &=\left[
        \frac{1-\rho}{2}\I_{\lX}
        +(1+\rho)\mathcal A
      \right]
      \left(\frac12\I_{\lX}+\mathcal A\right)^{-1}.
    \label{eq:SDE-scalar-factor}
\end{align}
Since the second factor is bijective, Theorem~4.4 gives
\begin{equation}\label{eq:SDE-scalar-bijective}
    T_{\rho\I_{\lX}}\text{ is bijective}
    \quad\Longleftrightarrow\quad
    \frac{1-\rho}{2}\I_{\lX}
    +(1+\rho)\mathcal A
    \text{ is bijective}.
\end{equation}
For $\rho\neq-1$, this is equivalently
\begin{equation}\label{eq:SDE-Robin-bijective}
    T_{\rho\I_{\lX}}\text{ is bijective}
    \quad\Longleftrightarrow\quad
    \mathcal A+\kappa_\rho\I_{\lX}
    \text{ is bijective},
    \qquad
    \kappa_\rho=\frac{1-\rho}{2(1+\rho)}.
\end{equation}
When $\rho\notin[-1,1]$, one has $\kappa_\rho<0$ and
$\rho\I_{\lX}$ is expansive, so bijectivity is no longer automatic from
the {(V)}-boundary theory. Formula
\eqref{eq:SDE-Robin-bijective} gives the exact criterion in that case.

\medskip
\noindent\textbf{Boundary triple.}
Define
\begin{equation}\label{eq:SDE-gamma}
    \gamma_0\begin{pmatrix}\mp\\u\end{pmatrix}:=\mT_0u,
    \qquad
    \gamma_1\begin{pmatrix}\mp\\u\end{pmatrix}:=R\mT_{\mnu}\mp.
\end{equation}
Then $(\lX,\gamma_1,\gamma_0)$ is a boundary triple. Its Green identity is
exactly \eqref{eq:SDE-boundary}, and the map
$(\gamma_1,\gamma_0)^\top$ is surjective because the normal and Dirichlet
trace variables arise from the two independent components of $\lW$.

\medskip
\noindent\textbf{{(M)}-boundary conditions.}
Let $\phi,\psi\in\mathcal L(\lX)$ be non-expansive and suppose that
$\I_{\lX}-\psi\phi$ is bijective. Set
\begin{equation}\label{eq:SDE-Cpsi}
    C_{\phi,\psi}
    :=
    \begin{pmatrix}
        \I_{\lX} & -\psi\\
        -\phi & \I_{\lX}
    \end{pmatrix}
    \begin{pmatrix}
        \I_{\lX} & \psi\\
        \phi & \I_{\lX}
    \end{pmatrix}^{-1}.
\end{equation}
The corresponding $M$-operator is characterised by
\begin{equation}\label{eq:SDE-M-general}
    \dup{\lW'}{M_{\phi,\psi}\vu}{\vv}{\lW}
    =\scp{C_{\phi,\psi}\Gamma\vu}{\Gamma\vv}_{\lX\oplus\lX},
    \qquad \vu,\vv\in\lW,
\end{equation}
where $\Gamma=(\Gamma_1,\Gamma_0)^\top$.

For the Neumann parameter $\phi=\I_{\lX}$, choose
$\psi=-\I_{\lX}$. Then
\[
    C_{\I,-\I}
    =\begin{pmatrix}
        0 & \I_{\lX}\\
        -\I_{\lX} & 0
      \end{pmatrix}.
\]
Writing
\[
    d_u:=\mT_0u,
    \qquad n_{\mp}:=R\mT_{\mnu}\mp,
\]
and analogously $d_v$ and $n_{\mq}$, we obtain
\begin{equation}\label{eq:SDE-M-neumann}
    \dup{\lW'}{M_{\I,-\I}\vu}{\vv}{\lW}
    =-\scp{n_{\mp}}{d_v}_{\lX}
     +\scp{d_u}{n_{\mq}}_{\lX}.
\end{equation}
This agrees with the boundary-operator construction used for the Neumann boundary condition in \cite[Section~6]{ABjde}. 

For the same Neumann parameter $\phi=\I_{\lX}$ and the choice $\psi=0$,
one obtains
\begin{equation}\label{eq:SDE-M-zero}
    \dup{\lW'}{M_{\I,0}\vu}{\vv}{\lW}
    =4\scp{n_{\mp}}{n_{\mq}}_{\lX}
     +\scp{d_u}{n_{\mq}}_{\lX}
     -\scp{n_{\mp}}{d_v}_{\lX}.
\end{equation}
In particular,
\[
    \Re\dup{\lW'}{M_{\I,0}\vu}{\vu}{\lW}
    =4\|n_{\mp}\|_{\lX}^2\geq0.
\]
Both operators correspond to the same Neumann {(V)}-domain:
\[
\ker(D-M_{\I,-\I})=\ker(D-M_{\I,0})
=\left\{\begin{pmatrix}\mp\\u\end{pmatrix}\in\lW:\mT_{\mnu}\mp=0\right\}.
\]
Thus the example exhibits both the computability of $Q_0$ and the multiplicity of $M$-operators associated with a fixed {(V)}-boundary condition.

\end{example}

\section{Concluding Remarks}
We have developed a boundary quadruple description of bijective realisations of abstract Friedrichs operators. The relation-level criterion
\[
    \lK=\Theta\dotplus\Gamma(\ker T_1)
\]
identifies the intrinsic geometric obstruction to bijectivity. For bounded graph parameters this obstruction is represented by the strict contraction
\[
    Q_0=\Gamma_1\bigl(\Gamma_0|_{\ker T_1}\bigr)^{-1},
\]
and the exact criterion becomes the invertibility of
$\I_{\lK_0}-\phi Q_0$. The canonical von Neumann parametrisation is recovered when $Q_0=0$, while arbitrary prescribed boundary quadruples generally require the additional factor.

The (V)-boundary conditions are parametrised by the non-expansive parameters and the self-orthogonal case by unitary parameters. For a fixed (V)-condition, compatible negative boundary spaces are characterised by the bijectivity of $\I_{\lK_1}-\psi\phi$. Choosing the orthogonal complement of the minimal space produces a canonical family of associated \emph{(M)}-operators. The ODE example computes the transfer operator explicitly, and the diffusion example identifies it with a Cayley transform of the Dirichlet-to-Neumann operator.

A natural continuation is to replace the distinguished point $0$ by a spectral parameter and to study a family $Q(\lambda)$. Such a construction should lead to Kre\u{\i}n-type resolvent formulas and to spectral criteria expressed through $\I-\phi Q(\lambda)$, in analogy with Weyl-function methods for symmetric operators and adjoint pairs.

\section*{Acknowledgements}
This work is supported by the Austrian Science Fund (FWF) through the ESPRIT Programme [Grant DOI: 10.55776/ESP1299024].

\end{document}